\documentclass[a4paper,12pt, reqno, twoside]{amsart}
\usepackage{amsfonts, color}
\usepackage{xcolor}
\usepackage{geometry}
\usepackage{amstext, amsthm, amssymb, amsmath}
\usepackage{mathrsfs}
\usepackage{graphicx}
\usepackage{color}
\usepackage{mathrsfs} 
\usepackage[linesnumbered,lined,ruled,commentsnumbered]{algorithm2e}
\usepackage{hyperref}

\numberwithin{equation}{section}
\theoremstyle{plain}
\begingroup
\newtheorem{theorem}{Theorem}[section]
\newtheorem{lemma}[theorem]{Lemma}
\newtheorem{proposition}[theorem]{Proposition}
\newtheorem{corollary}[theorem]{Corollary}
\endgroup

\theoremstyle{definition}
\begingroup
\newtheorem{defn}{Definition}

\newtheorem{assumption}{Assumption}
\endgroup

\theoremstyle{remark}
\newtheorem{remark}{Remark}

\newcommand{\R}{\mathbb{R}}
\newcommand{\F}{\mathcal{F}}
\newcommand{\Prob}{\mathcal{P}}
\newcommand{\U}{\mathcal{U}}

\newcommand{\e}{\varepsilon}
\newcommand{\weak}{\rightharpoonup}
\newcommand{\G}{\mathcal{G}}

\newcommand{\NN}{\mathbb{N}}

\newcommand{\X}{\mathcal{X}}
\newcommand{\opt}{\mathrm{Opt}}
\newcommand{\Id}{\mathrm{Id}}
\newcommand{\supp}{\mathrm{supp}}

\begin{document}
\title[Normalizing flows as approximations of optimal transport maps]
{Normalizing flows as approximations of optimal transport maps via linear-control neural ODE\lowercase{s}}

\author[A. Scagliotti]{Alessandro Scagliotti}
\author[S. Farinelli]{Sara Farinelli}

\address[A.~Scagliotti]{}
\address{Technische Universit\"at M\"unchen, Garching b. M\"unchen, Germany}
\address{Munich Center for Machine Learning (MCML), Munich, Germany \vspace{0.25cm}}
\email{scag@ma.tum.de \vspace{0.5cm}}

\address[S.~Farinelli]{}
\address{DIMA-MalGa, University of Genoa, 16146 Genoa, Italy \vspace{0.25cm}}
\email{sara.farinelli@edu.unige.it}

\begin{abstract}
In this paper, we consider the problem of recovering the $W_2$-optimal transport
map T between absolutely continuous measures $\mu,\nu\in\mathcal{P}(\mathbb{R}^n)$ as the flow of a
linear-control neural ODE, where the control depends only on the time variable
and takes values in a finite-dimensional space. 
 We first show that, under suitable assumptions on $\mu,\nu$ and on the controlled vector fields governing the neural ODE, 
the optimal transport map is contained in the $C^0_c$-closure of the flows generated by the system. 
Then, we tackle the problem under the assumption that  only discrete approximations of 
$\mu_N,\nu_N$ of the original measures $\mu,\nu$ are available: we formulate approximated optimal control problems, and we show that their solutions 
give flows that approximate the original optimal transport map $T$. In the framework of generative models, the approximating flow constructed here can be seen as a `Normalizing Flow', which usually refers to the task of providing invertible transport maps between probability measures by means of deep neural networks.
We propose an iterative numerical scheme based on the Pontryagin Maximum Principle
for the resolution of the optimal control problem, resulting in a method for the practical computation of the approximated optimal transport map, and we test it on a two-dimensional example.
\vspace{12pt}

\subsection*{Keywords:} $\Gamma$-convergence, Optimal control, Optimal transport, Linear-control neural ODEs.

\subsection*{Mathematics Subject Classification:}
34H05, 49Q22, 49J45, 49M05.
\end{abstract}

\maketitle

\section*{Introduction}
\renewcommand{\theequation}{I.\arabic{equation}}

In this paper, we consider the problem of approximating the optimal transport map between compactly-supported probability measures in $\R^n$ by means of flows induced by linear-control systems. Namely, we consider controlled dynamical systems of the form
\begin{equation}\label{eq:intro_lin_ctrl}
    \dot x(t) = F(x(t))u(t) =
    \sum_{i=1}^k F_i(x(t)) u_i(t)
    \qquad \mbox{a.e. } t\in [0,1],
\end{equation}
where $F=(F_1,\ldots,F_k):\R^n \to \R^{n\times k}$ defines the controlled vector fields, and $u\in \U:=L^2([0,1],\R^k)$ is the control, which takes values in a finite-dimensional space and depends only on the time variable (i.e.~it is \emph{open loop}). 
The term `linear-control' indicates the linear dependence of the system in the controls,
{ which in turn guarantees that}
setting the time horizon as $[0,1]$ is not restrictive.
In our case, the object of interest is the diffeomorphism $\Phi_u:\R^n\to \R^n$, obtained as the terminal-time flow associated to \eqref{eq:intro_lin_ctrl} and corresponding to $u\in\U$. 
In particular, given two probability measures $\mu,\nu\in \mathcal{P}(\R^n)$ with compact support and denoting with $T:\supp(\mu)\to\supp(\nu)$ the optimal transport map with respect to the the $W_2$-distance, we aim at approximating $T$ with elements in  $\mathscr{F}:= \{ \Phi_u \mid u\in \U \}$.
The starting point of our analysis is represented by the controllability results obtained in \cite{AS1,AS2}. Here, the authors formulated the notion of \emph{Lie Algebra strong approximating property}, and they showed that, if the vector fields $F_1,\ldots,F_k$ satisfy it, then the flows in $\mathscr F$ are dense in the $C^0_c$-topology in the class of the diffeomorpisms isotopic to the identity.
In the first part of this work, we use the classical regularity theory of Monge Ampère equation (\cite{Ca1, Ca2}) to prove that, under suitable assumptions on $\mu,\nu$ and their densities, the $W_2$-optimal transport map $T$ is a diffeomorphism isotopic to the identity (Proposition~\ref{prop:isotopic}), paving the way to the approximation of $T$ through the flows contained in $\mathscr F$ (Corollary~\ref{cor:ot_map_approx}).

From a practical perspective, the most interesting scenario is the reconstruction of the optimal transport map when it is not explicitly known.
For example, in a \emph{data-driven approach}, one or both measures $\mu,\nu$ may be not directly available, and we may have access only to discrete approximations $\mu_N,\nu_N$, obtained, e.g., through empirical samplings. 
In this context, we mention the recent advances in  statistical optimal transport, and we refer the interested reader to \cite{LGL_21, Rigollet_21, MBNW_21}.
We also report the contribution \cite{PCFH_16}, where the authors propose an algorithm to \emph{learn} at the same time  an optimal coupling between $\mu_N,\nu_N$ and an approximated optimal transport map.  
In this paper, our goal consists in approximating the optimal transport map $T$ starting from a discrete optimal coupling $\gamma_N$ between $\mu_N$ and $\nu_N$.
Namely, using the flows induced by \eqref{eq:intro_lin_ctrl}, we define the functional $\F^{N,\beta}:\U\to\R$ as 
\begin{equation}\label{eq:intro_def_funct_N}
    \F^{N,\beta}(u):=
    \int_{\R^n\times \R^n} |\Phi_u(x) - y|_2^2 \,d\gamma_N(x,y)
    + \frac\beta2 \|u\|_{L^2}^2,
\end{equation}
where $\beta>0$ is a parameter that tunes the $L^2$-regularization, which is essential to provide coercivity.
In Theorem~\ref{thm:OT_map_appr}, we prove that, when $\mu_N\weak^*\mu$ and $\nu_N\weak^*\nu$ as $N\to\infty$, assuming that $\mu\ll \mathcal{L}_{\R^n}$, the sequence of functionals $(\F^{N,\beta})_N$ is $\Gamma$-convergent with respect to the $L^2$-weak topology to the functional
\begin{equation}\label{eq:intro_def_funct_lim}
    \F^{\infty, \beta}(u):=
    \int_{\R^n}|\Phi_u(x)-T(x)|_2^2 \, d\mu(x) + \frac\beta2
    \|u\|_{L^2}^2,
\end{equation}
where $T$ is the optimal transport map,  from $\mu$ to $\nu$.
Moreover, under the hypotheses that ensure that $T$ is contained in the closure of $\mathscr F$, it turns out that every minimizer $\hat u$ of $\F^{\infty,\beta}$ generates a flow $\Phi_{\hat u}$ that can be made arbitrarily close to $T$ in the $L^2_\mu$-norm, by setting $\beta$ small enough.
In this framework, the $\Gamma$-convergence result guarantees that, in practical applications where we deal with the discrete measures $\mu_N,\nu_N$, we can minimize \eqref{eq:intro_def_funct_N} in place of \eqref{eq:intro_def_funct_lim}. In fact, it is interesting to mention that the minimizers of $\F^{N,\beta}$ converge to the minimizers of $\F^{\infty,\beta}$ in the \emph{$L^2$-strong} topology, and not just in the weak sense.
This is due to the fact that, being the system \eqref{eq:intro_lin_ctrl} linear in the controls, the integral term in \eqref{eq:intro_def_funct_N}--\eqref{eq:intro_def_funct_lim} is continuous with respect to the $L^2$-weak convergence of the controls.
This property has been recently exploited also in \cite{S_deep,S_ens}, in problems related to diffeomorphisms approximation and simultaneous control of ensembles of systems, respectively.
The present paper can be read as a generalization of the approach proposed in \cite{S_deep}, where the task consisted in \emph{learning} an unknown diffeomorphism $\Psi:\R^n\to\R^n$ through a linear-control system. In \cite{S_deep}, the \emph{training data-set} was represented by the collections of observations $\{(x_j,y_j=\Psi(x_j))\}_{j=1,\ldots,N}$, with a clear and assigned bijection between the initial points $\{ x_j\}_{j=1,\ldots,N}$ and the targets $\{ y_j\}_{j=1,\ldots,N}$.
In the present situation, if we set $\supp(\mu_N):=\{x_1,\ldots,x_{N_1}\}$ and $\supp(\nu_N):=\{y_1,\ldots,y_{N_2}\}$, we cannot expect \emph{a priori} a bijection between the elements of the supports. However, a $W_2$-optimal transport plan $\gamma_N$ from $\mu_N$ to $\nu_N$ provides us with a \emph{weighted} correspondence between the supports, that we employ to formulate \eqref{eq:intro_def_funct_N}.
{Finally, it is worth mentioning that our approach can be pursued as well even when the coupling $\gamma_N$ has not been obtained by solving the discrete optimal transport problem between $\mu_N$ and $\nu_N$, as observed in Remark~\ref{rmk:nonopt_plans}.
}

In the last decades optimal transport has been employed in many applied mathematical fields, such as  Machine Learning \cite{CR12,FHS23}, {generative models} \cite{ACB,RKB}, and signal and data analysis \cite{ABGV_18,Thorpe_17}, to mention a few. 
Our investigation is closely related to a problem that, in the context of {generative models}, is known in the Machine Learning literature as \emph{Normalizing Flows}.
Namely, given $\mu,\nu \ll \mathcal{L}_{\R^n}$ with densities $\rho_\mu,\rho_\nu:\R^n\to \R_+$, the task consists in finding a change of variable, i.e.\ an invertible and differentiable map $\phi_{\mathbf{u}}:\R^n\to\R^n$ such that
\begin{equation}\label{eq:intro_NF}
    \rho_\nu(y) \approx \rho_\mu(\phi_\mathbf{u}^{-1}(y))
    \left|\det \nabla \phi_\mathbf{u} \big(\phi_\mathbf{u}^{-1}(y) \big) \right|^{-1},
\end{equation}
where $\mathbf{u}= (u_1,\ldots,u_L)\in \R^{d\times L}$, and $\phi_\mathbf{u}$ is a \emph{deep neural network} expressed as the composition of $L$ parametric elementary mappings (\emph{layers}) $\varphi_{u_1},\ldots,\varphi_{u_L}:\R^n\to\R^n$, i.e., 
$\phi_\mathbf{u}=\varphi_{u_L}\circ\ldots\circ\varphi_{u_1}$.
The tuning of the parameters $u_1,\ldots,u_L$ (\emph{training}) is performed by log-likelyhood maximization of \eqref{eq:intro_NF}.
For further details on this topic, we refer the reader to the review papers \cite{Papa_nf_21, Koby_nf_20}.
In the seminal works \cite{WE17,haber17} it was established a fundamental connection between Deep Learning and Control Theory, so that deep neural networks can be effectively modeled by control systems. This approach has been popularized in \cite{CRBD18} under the name \emph{neural ODEs}, and it is crucial for current development and understanding of Machine Learning (see, e.g., \cite{BCFH_23,CFS_23,EGPZ,RBZ_nODE}). 
In our formulation, the system \eqref{eq:intro_lin_ctrl} plays the role of a linear-control neural ODE.
In the framework of neural ODEs, the problem of Normalizing Flows has been recently tackled from a controllability perspective in \cite{RBZ_NF}, where the authors consider a nonlinear-control system and propose an explicit construction for the controls, so that the corresponding final-time flow is an approximate transport map between two assigned absolutely continuous measures $\mu,\nu\in \mathcal{P}(\R^n)$. We report that the maps obtained in \cite{RBZ_NF} are not aimed at being optimal.
Finally, in \cite{FGOP_20} the computation of a normalizing flow is carried out by learning Entropy-Kantorovich potentials, and in \cite{MDBCR_23} it is proposed a post-processing for trained normalizing flows to reduce their transport cost.
{We insist on the fact that the controls $u\in L^2([0,1],\R^k)$ considered in this paper take values in finite-dimensional spaces, as it is as well the case in \cite{AS1, AS2}, where the controllability results we rely on were established. 
On the other hand, in \cite{AgCapo, Capo}, the authors had previously investigated the controllability problem in the group of diffeomorphisms when allowing the controls to depend on the state-variable, i.e.~to have values in infinite-dimensional spaces.
The latter viewpoint has been fruitfully adopted in the framework of shape deformations \cite{Y19}, in particular with applications to imaging problems (see e.g.~\cite{ATTY, Via12}). }

This paper is organized as follows.\\
In Section~\ref{sec:prel}, we establish our notations and we collect some basic results in Optimal Transport and Control Theory, respectively. \\
In Section~\ref{sec:appr_OT_map}, we show that, under proper regularity assumptions on the measures $\mu,\nu$ and their densities, the $W_2$-optimal transport map is a diffeomorphism isotopic to the identity (Proposition~\ref{prop:isotopic}), and it is approximable with a flow induced by a linear-control system (Corollary~\ref{cor:ot_map_approx}). \\
In Section~\ref{sec:G-conv}, we establish the $\Gamma$-convergence result for the functionals $\F^{N,\beta}$ defined as in \eqref{eq:intro_def_funct_N} (Theorem~\ref{thm:G_conv}), working in a slightly more general setting than the remainder of the paper. In Theorem~\ref{thm:OT_map_appr} we focus our attention to the main problem of the paper, i.e., the recovery of the optimal transport map. Moreover, in Remark~\ref{rmk:triang_ineq_app_OT} we provide an asymptotic estimate for $N$ large of $W_2(\Phi_{\hat u\, \#}\mu, \nu)$ with $\hat u \in \arg \min \F^{N,\beta}$, and in Remark~\ref{rmk:geodesics} we discuss the possibility of approximating the $W_2$-geodesic connecting $\mu$ to $\nu$. \\
Finally, in Section~\ref{sec:num_approx}, we propose a numerical scheme for the approximate minimization of the functionals $\F^{N,\beta}$ based on the Pontryagin Maximum Principle.
In fact, this results in an algorithm for reconstructing the optimal transport map between $\mu,\nu$ by using an optimal coupling $\gamma_N$ between the empirical measures $\mu_N,\nu_N$.
We perform an experiment in $\R^2$ to validate the theoretical results.

\renewcommand{\theequation}{\thesection.\arabic{equation}}

\section{Preliminaries and Notations} \label{sec:prel}

\subsection{Preliminaries on Optimal Transport} Here, we collect some basic facts in Optimal Transport which will be useful for our purposes. We refer the reader to \cite{AG, Sant, Vi} for a complete introduction to the topic. 
For any $n\geq 1$ we denote by $\mathcal{P}(\R^n)$ the set of Borel probability measures on $\R^n$. 
We recall some definitions and basics facts about probability measures. 
\begin{defn}
Given  a Borel probability measure $\mu\in \mathcal{P}(\R^n)$ and a Borel map $T:\R^n\to \R^{n'}$ then the \emph{pushforward measure} of $\mu$ through the map $T$ is defined as the measure $T_{\sharp}\mu\in \mathcal{P}(\R^{n'})$ such that for any $A$ Borel set of $\R^{n'}$
\begin{equation*}
T_{\sharp}\mu(A):=\mu(T^{-1}(A)),
\end{equation*}
where $T^{-1}(A)$ is the preimage of $A$ through the map $T$. 
\end{defn}
The
pushforward measure can be characterized 
by means of the following identity:
\begin{equation}\label{eq:char_pushforward}
\int_{\R^{n'}}\varphi(x)\,d T_{\sharp}\mu(x)=\int_{\R^n} \varphi\circ T(x)\,d\mu(x)
\end{equation}
for every $\varphi\in C^{0}_{b}(\R^{n'}, \R)$.

We recall the notion of weak convergence
of probability measures.
\begin{defn} \label{def:weak_conv_prob}
For every $n\geq1$, we say that the sequence $(\eta_N)_{N\geq1}\subset \mathcal{P}(\R^n)$
is weakly convergent to $\eta_\infty\in \mathcal{P}(\R^n)$ if for every continuous bounded
function $\varphi\in C_b^0(\R^n,\R)$ the following identity holds:
\begin{equation*}
    \lim_{N\to\infty} \int_{\R^n} \varphi(x)d\eta_N(x) = \int_{\R^d} \varphi(x)d\eta_\infty(x),
\end{equation*}
and we write $\eta_N\weak^*\eta_\infty$ as $N\to\infty$.
\end{defn}

In the next result we recall that the 
pushforward trough continuous maps is 
stable with respect to the weak convergence.
\begin{lemma}\label{lem:stability_under_pushfw}
Let $(\mu_N)_{N\geq 1}$ be a sequence of probability measures of  $\R^n$ and $\mu_{\infty}\in\mathcal{P}(\R^n)$ such that $\mu_N\weak^*\mu_{\infty}$ as $N\to +\infty$. Let $T:\R^n\to\R^{n'}$ be a continuous map.  Then  $T_{\sharp}\mu_N\weak^*T_{\sharp}\mu_{\infty}$ as $N\to +\infty$.
\end{lemma}
\begin{proof}
It descends immediately from \eqref{eq:char_pushforward}, Definition~\ref{def:weak_conv_prob}, and the fact that that $\varphi\circ T\in C^{0}_{b}(\R^{n'}, \R) $ if  $\varphi\in  C^{0}_{b}(\R^n, \R)$.
\end{proof}
We denote by $\mathcal{P}_2(\R^n)$ the set of Borel probability measures having finite second moment, namely
\begin{equation*}
\mathcal{P}_2(\R^n):=\left\{ \mu \in \mathcal{P}(\R^n)\colon\, \int_{\R^n}|{x}|^2\, d\mu(x)<+\infty
\right\}.
\end{equation*}
For any two probability measures $\mu,\nu\in \mathcal{P}(\R^n)$ we define the set of \emph{admissible transport plans} between $\mu$ and $\nu$ as 
\begin{equation*}
\mathrm{Adm}(\mu,\nu):=\lbrace \gamma \in \mathcal{P}(\R^n\times\R^n)\colon (P_1)_{\sharp}\gamma=\mu,\, (P_2)_{\sharp}\gamma=\nu \rbrace, 
\end{equation*}
where $P_1, P_2:\R^n \times \R^n\to\R^n$ are the canonical projections 
on the first and second component, respectively.
\begin{defn}
For any two probability measures $\mu,\nu\in \mathcal{P}_2(\R^n)$, the \emph{$2$-Wasserstein distance} between $\mu$ and $\nu$ is defined as follows:
\begin{equation}\label{def:W2}
W_2(\mu,\nu):=\left(
\inf  \left\{
\int_{\R^n\times\R^n}|x-y|^2
\,d\gamma(x,y)\colon
\gamma \in \mathrm{Adm}(\mu,\nu)
\right\}
\right)^{\frac12}
\end{equation}
\end{defn}
We denote by $\opt(\mu,\nu)$ the set of admissible plans which realize the infimum in \eqref{def:W2}:
\begin{equation} \label{eq:def_opt_plans} 
\opt(\mu,\nu):=\left\{ \gamma \in \text{Adm}(\mu,\nu)\colon \int_{\R^n\times\R^n}|x-y|^{2}
\,d\gamma(x,y)=W^2_2(\mu,\nu) \right\}.
\end{equation}
It follows from classical arguments that the set $\opt(\mu,\nu)$ is non empty (see e.g. \cite[Theorem 1.5]{AG}).
We say that a Borel map $T:\R^n\to \R^n$ is an \emph{optimal transport map} between $\mu,\, \nu\in\mathcal{P}_2(\R^n)$ if 
$\gamma_T:=(\Id,T)_{\sharp}\mu\in \opt(\mu,\nu)$. We emphasize that in this paper we shall use the term \emph{optimal transport map} only referring to the cost related to the Euclidean squared distance.

We remark that if $(\eta_N)_{N\geq1}$ is a sequence of  probability measures with supports contained on a compact set $K\subseteq \R^d$, then the sequence weakly converges to a probability measure $\eta_{\infty}$ in the sense of Definition \ref{def:weak_conv_prob} if and only if $\lim_{N\to +\infty}W_2(\eta_N,\eta_{\infty})=0$, i.e. it converges in the $2$-Wasserstein distance (see e.g. \cite[Theorem 5.10]{Sant}).
\begin{proposition} \label{prop:conv_opt_plans}
Let $(\mu_N)_{N\geq1}, (\nu_N)_{N\geq1}\subset \mathcal{P}(\R^n)$ and be two sequences of probability measures, 
and let $\mu_\infty, \nu_\infty\in\mathcal{P}(\R^n) $ be such 
that $\mu_N\weak^* \mu_\infty$ and $\nu_N\weak^* \nu_\infty$ as $N\to\infty$. 
Let  $(\gamma_N)_{N\geq1}\subset \mathcal{P}(\R^n\times\R^n)$ be a sequence of probability measures satisfying  $(\gamma_N)_{N\geq1}\in\opt(\mu_N,\nu_N)$ for every $N\geq1$.
Then the sequence $(\gamma_N)_{N\geq1}$ is
weakly pre-compact, and every 
limiting point belongs to
$\opt(\mu_\infty,\nu_\infty)$.
\end{proposition}
\begin{proof}
See \cite[Proposition~2.5]{AG}.
\end{proof}


\subsection{Preliminaries on linear-control systems} \label{subsec:lin-ctrl}
In this section, we present some classical results for linear-control system that will be useful in the rest of the paper. 
We consider controlled dynamical systems in $\R^n$ of the form
\begin{equation}\label{eq:lin_ctrl}
    \dot x(t) = F(x(t))u(t) = \sum_{i=1}^k F_i(x(t))
    u_i(t) \qquad \mbox{a.e. in } [0,1],
\end{equation}
where $F=(F_1,\ldots,F_k):\R^n\to\R^{n\times k}$
is a smooth matrix-valued application that defines the control system, and $u=(u_1,\ldots,u_k)\in L^2([0,1],\R^k)$ is the control.
We assume the controlled vector fields $F_1,\ldots,F_k$ to be Lipschitz-continuous, i.e., there exists a constant $L>0$ such that
\begin{equation} \label{eq:lipsch_fields}
    \sup_{i=1,\ldots,k} \, \sup_{x\neq y} \,
    \frac{|F_i(x)-F_i(y)|_2}{|x-y|_2} \leq L.
\end{equation}
From the previous condition, it follows that the vector fields $F_1,\ldots,F_k$ have sub-linear growth, i.e., there exists $C>0$ such that
\begin{equation}\label{eq:sub_growth}
    |F_i(x)|_2 \leq C(1+|x|_2)
\end{equation}
for every $x\in \R^n$ and for every $i=1,\ldots,k$.
We denote by $\U:=L^2([0,1],\R^k)$ the space of admissible controls, and we endow it with the usual Hilbert space structure induced by the scalar product defined as
\begin{equation}
    \langle u, v\rangle_{L^2} :=
    \int_0^1 \langle u(t),v(t) \rangle_{\R^k}\, dt
\end{equation}
for every $u,v\in\U$.
For every $u\in\U$ we consider the diffeomorphism $\Phi_u:\R^n\to\R^n$ defined as
\begin{equation} \label{eq:def_Phi_u}
    \Phi_u(x) := x_u(1)
\end{equation}
for every $x\in\R^n$, where the absolutely continuous curve $x_u:[0,1]\to\R^n$ solves the Cauchy problem
\begin{equation} \label{eq:Cau_ctrl}
    \begin{cases}
        \dot x_u(t) = F(x_u(t))u(t) &\mbox{a.e. in } 
        [0,1],\\
        x_u(0) = x.
    \end{cases}
\end{equation}
We recall that the existence and uniqueness of the solution of \eqref{eq:Cau_ctrl} is guaranteed by Carath\'eodory Theorem (see, e.g., \cite[Theorem~5.3]{H80}).
We observe that considering the time span equal to $[0,1]$ in \eqref{eq:Cau_ctrl} is not restrictive for our purposes. Indeed, using the fact that the dynamics is linear in the controls, given a general evolution horizon $[0,T]$ with $T>0$, we can always reduce to the case $[0,1]$ by rescaling the controls.
We now investigate the Lipschitz continuity of the flows generated by the linear-control system \eqref{eq:lin_ctrl}.

\begin{lemma}\label{lem:Lip_flow}
For every $u\in\U$, let $\Phi_u:\R^n\to\R^n$ be the 
flow defined as in \eqref{eq:def_Phi_u},
associated to the linear-control system 
\eqref{eq:lin_ctrl} and corresponding to
the admissible control $u$.
For every $\rho>0$ there exists a $L'>0$ 
such that
\begin{equation} \label{eq:Lip_flow}
    |\Phi_u(x^1)-\Phi_u(x^2)|_2 \leq L'
    |x^1-x^2|_2
\end{equation}
for every $x^1,x^2\in\R^n$ and for every
$u\in\U$ with $\|u\|_{L^2}\leq\rho$.
\end{lemma}
\begin{proof}
    See \cite[Lemma~2.3]{S_deep} or in Appendix~\ref{app:lin-ctrl}.
\end{proof}

We conclude this section by recalling a convergence result.

\begin{proposition}\label{prop:conv_flows}
Let us consider a sequence $(u_m)_{m\in\NN}\subset \U$
and $u_\infty\in\U$ such that $u_m\weak_{L^2} u_\infty$ as
$m\to\infty$. For every $m\in\NN\cup\{\infty \}$, 
let $\Phi_{u_m}:\R^n\to\R^n$ be the flow
generated by the control system \eqref{eq:lin_ctrl}
and corresponding to the admissible control
$u_m$.
Then, for every compact set $K\subset \R^n$, we have that
\begin{equation} \label{eq:conv_flows}
    \lim_{m\to\infty} \sup_{x\in K}
    |\Phi_{u_m}(x) - \Phi_{u_\infty}(x)  |_2 =0.
\end{equation}
\end{proposition}
\begin{proof}
    See \cite[Proposition~2.4]{S_deep} or in Appendix~\ref{app:lin-ctrl}.
\end{proof}

\begin{remark}\label{rmk:lin-ctrl_weak_conv}
    In the previous proposition the fact that the system is linear in the control variables plays a crucial role. Indeed, in the case of a nonlinear-control system (or neural ODE)
    \begin{equation*}
        \dot x = G(x,u),
    \end{equation*}
    in general it is not true that \textit{weakly-convergent} controls result in flows converging uniformly over compact subsets. In this situation, the local convergence of the flows holds if the controls are \textit{strongly} convergent. However, equipping the space of admissible controls with the $L^2$-strong topology is not suitable for our $\Gamma$-convergence argument.     
\end{remark}


\section{Approximability of the optimal transport map} \label{sec:appr_OT_map}
In this section, we address the problem of approximating the optimal transport map using flows generated by a linear-control system \eqref{eq:lin_ctrl}, where the controlled vector fields $F_1,\ldots,F_k$ satisfy a proper technical condition.
We begin by reporting some results concerning the approximation capabilities of flows generated by this kind of systems. We refer the interested reader to \cite{AS1,AS2} for a detailed discussion in full-generality.

We recall the definition of Lie algebra generated by a system of vector fields. Given the vector fields $F_1,\ldots,F_k$, the linear space $\mathrm{Lie}(F_1,\ldots,F_k)$ is defined as
\begin{equation*} \label{eq:Lie_gen}
\mathrm{Lie}(F_1,\ldots,F_k):=\mathrm{span}
\{ [F_{i_s},[\ldots ,[F_{i_2},F_{i_1}],\ldots]]
: s\geq1, i_1,\ldots,i_s\in \{ 1,\ldots,k \}\},
\end{equation*}
where $[F,F']$ denotes the Lie bracket between the smooth vector fields $F,F'$ of $\R^n$.
In view of the main result, we need to consider the subset of the Lie algebra generated by $F_1,\ldots,F_k$ whose vector fields have bounded $C^1$-norm on compact sets of $\R^n$. Given a vector field $X:\R^n\to\R^n$ and a compact set $K\subset \R^n$, we define
\begin{equation*}    
\|X\|_{1,K}:= \sup_{x\in K}\left(|X(x)|_2 + 
\sum_{i=1}^n|D_{x_i}X(x)|_2 \right). 
\end{equation*}
Finally, we introduce
\begin{equation*}
\mathrm{Lie}^\delta_{1,K}(F_1,\ldots,F_k):=
\{
X\in \mathrm{Lie}(F_1,\ldots,F_k):
\|X\|_{1,K}\leq \delta
\}.
\end{equation*}
We now formulate the assumption required
for the approximability result.
\begin{assumption} \label{ass:str_Lie}
The system of vector fields $F_1,\ldots,F_k$ satisfies
the {\it Lie algebra strong approximating property},
i.e., there exists $m\geq1$ such that, for every
$C^m$-regular vector field $Y:\R^n\to\R^n$ and for every compact set $K\subset\R^n$, there exists $\delta>0$ such that
\begin{equation} \label{eq:Lie_alg_cond}
\inf \left\{ \sup_{x\in K}|X(x)-Y(x)|_2 \,\, \Big|
X\in \mathrm{Lie}_{1,K}^\delta (F_1,\ldots,F_k) \right
\}=0.
\end{equation}
\end{assumption}

The next result illustrates the powerful approximation capabilities of flows of linear-control systems whose fields fulfill Assumption~\ref{ass:str_Lie}.

\begin{theorem}\label{thm:univ_app_flows}
Let $\Psi:\R^n\to\R^n$ be a diffeomorphism isotopic to the identity. 
Let $F_1,\ldots,F_k$ be a system of vector fields satisfying Assumption~\ref{ass:str_Lie}. Then, for each compact set $K\subset\R^n$ and each $\e>0$ there exists an admissible control $u\in\U$ such that 
\begin{equation}\label{eq:Univ_app}
\sup_{x\in K}|\Psi(x)-\Phi_u(x)|_2\leq \e,
\end{equation}
where $\Phi_u$ is the flow corresponding to
the control $u$ defined in \eqref{eq:def_Phi_u}.
\end{theorem}
\begin{proof}
See \cite[Theorem~5.1]{AS2}.
\end{proof}

\begin{remark}\label{rmk:controll_result}
    We recall that a diffeomorphism $\Psi:\R^n\to\R^n$ is \textit{isotopic to the identity} if it can be expressed as the final-time flow induced by a non-autonomous vector field which is smooth in the state-variable. In other words, if there exists a time-varying vector field $Y:[0,1]\times \R^n \to \R^n$ such that $Y(t,\cdot)\in C^\infty (\R^n,\R^n)$ for every $t\in [0,1]$, and such that for every $x_0\in \R^n$ we have
    \begin{equation} \label{eq:isotopic}
        \Psi(x_0) = x(1), \quad \mbox{where} \quad
        \begin{cases}
            \dot x(t) = Y(t,x(t)) & t\in [0,1],\\
            x(0)=x_0.
        \end{cases}
    \end{equation}
    We observe that, by definition, any diffeomorphism $\Phi_u$ with $u\in \U$ of the form \eqref{eq:def_Phi_u} is isotopic to the identity. 
    The remarkable fact conveyed by Theorem~\ref{thm:univ_app_flows} is that, when Assumption~\ref{ass:str_Lie} holds, the family $\mathscr F :=\{ \Phi_u: u\in \U \}$ is dense with respect to the $C^0_c$-topology in the class of the diffeomorphisms isotopic to the identity.
    In the jargon of the Machine Learning community, Theorem~\ref{thm:univ_app_flows} can be classified as a \textit{universal approximation result}.
\end{remark}

\begin{remark} \label{rmk:functional_diffeo}
    Given a compact set $K\subset \R^n$, a probability measure $\mu\in \mathcal{P}(K)$ and a diffeomorphism $\Psi:\R^n\to \R^n$ isotopic to the identity, we can consider the functional $\F^{\mu, \beta}:\U\to \R_+$ defined as follows:
    \begin{equation} \label{eq:def_funct_diffeo}
        \F^{\mu, \beta}(u) := 
        \int_K | \Phi_u(x)- \Psi(x) |_2^2 \, d\mu(x) 
        + \frac\beta2 \|u \|_{L^2}^2,
    \end{equation}
    where $\beta>0$ is a parameter tuning the Tikhonov regularization on the energy of the control.
    The problem concerning the minimization of \eqref{eq:def_funct_diffeo} has been studied in detail in \cite{S_deep}. In particular, in virtue of the controllability result expressed in Theorem~\ref{thm:univ_app_flows}, it is possible to show that, for every $\epsilon > 0$, there exists $\bar \beta>0$ such that, for every $\bar u \in \arg\min_{\U} \F^{\mu, \bar \beta}$, we have
    \begin{equation*}
        \int_K | \Phi_{\bar u}(x)- \Psi(x) |_2^2 \, d\mu(x) \leq \epsilon.
    \end{equation*}
    For the details, see \cite[Proposition~5.4]{S_deep}. 
    The fact that, when $\beta$ is small enough, the minimizers of $\F^{\mu, \beta}$ achieve an arbitrarily small mean squared approximantion error is of primary importance for practical purposes.
    Indeed, even though the proof of Theorem~\ref{thm:univ_app_flows} in \cite{AS2} provides an explicit procedure to obtain the approximating flow, it requires the knowledge of a non-autonomous vector field $Y:[0,1]\times \R^n \to \R^n$ related to the fact that $\Psi$ is isotopic to the identity (see \eqref{eq:isotopic}).
    In addition, the control constructed with the strategy illustrated in \cite{AS2} cannot be expected to be optimal in the $L^2$-norm, among all the other controls that achieve the same quality of approximation.
    For this reason, in \cite{S_deep} the computational approximation of $\Psi$ was performed via the numerical minimization of \eqref{eq:def_funct_diffeo}.
\end{remark}

\begin{remark}\label{rmk:hermite_fields}
    We exhibit here a system of vector fields in $\R^n$ for which Assumption~\ref{ass:str_Lie} holds. For every $n>1$ and $\nu>0$, consider the vector fields in $\R^n$ 
    \begin{equation}\label{eq:v_fields_control}
    {\bar F_i}(x) := \frac{\partial}{\partial x_i},
    \quad 
    {\bar F'_i}(x) := e^{-\frac{1}{2\zeta}|x|^2}\frac{\partial}{\partial x_i},
    \quad i=1,\ldots,n,
    \end{equation}
    with $\zeta>0$.
    Then the system $\bar F_1,\ldots,\bar F_n,\bar F'_1,\ldots,\bar F'_n$ satisfies Assumption~\ref{ass:str_Lie} (see \cite[Proposition~6.1]{AS2}). The key-observation is that, by taking the Lie brackets of \eqref{eq:v_fields_control}, it is possible to generate the Hermite monomials of every degree.
    Therefore, any linear-control system having at least \eqref{eq:v_fields_control} among the controlled fields can generate flows with the approximation capabilities described by Theorem~\ref{thm:univ_app_flows}.
    Moreover, adding extra controlled fields to the family \eqref{eq:v_fields_control} is not going to improve Theorem~\ref{thm:univ_app_flows}, since, as explained above in Remark~\ref{rmk:controll_result}, the density result stated there is the best that one can expect.
    Even though this argument is correct from a theoretical viewpoint, it is interesting to observe that, for practical purposes, enlarging the family of vector fields \eqref{eq:v_fields_control} can be very beneficial. For further details on this intriguing point, we recommend the discussion in \cite[Remark~3.15]{S_deep} and the numerical experiments in \cite[Section~8]{S_deep}.
\end{remark}

 We conclude this section by showing that, under suitable assumptions on the probability measures $\mu, \nu$, the optimal transport map between $\mu$ and $\nu$ is a diffeomorphism isotopic to the identity.
\begin{proposition} \label{prop:isotopic}
Let $\mu=\rho_\mu\mathcal{L}_{\R^n}$ and $\nu=\rho_\nu\mathcal{L}_{\R^n}$ be two probability measures,  with $\rho_\mu:\Omega_1\to \R$ and $\rho_\nu:\Omega_2\to \R$, where $\Omega_1$ and $\Omega_2$ are open and bounded substets of $\R^n$. Let us assume  that there exist a constant $C>1$ such that $ C\geq \rho_\mu\geq {1}/{C}$ on $\Omega_1$ and $ C\geq \rho_\nu\geq {1}/{C}$ on $\Omega_2$, and in addition that \begin{itemize}
\item  $\rho_\mu\in C^{\infty}(\bar\Omega_1,\R^d) $ and $\rho_\nu\in C^{\infty}(\bar\Omega_2,\R^n)$;
\item $\Omega_1,\, \Omega_2$ are smooth and uniformly convex. 
\end{itemize}
Let $T:\bar \Omega_1\to \bar \Omega_2$ be the optimal transport map between
$\mu$ and $\nu$. Then $T$ is the restriction of a diffeomorphism isotopic to the identity. 
\end{proposition}
\begin{proof}
We proceed in four steps: in the first three we construct a smooth vector field, and in the last one we use this vector field to show that the optimal transport map $T: \bar \Omega_1 \to \bar \Omega_2$ is isotopic to the identity. 
We first make some preliminary observations. By Brenier Theorem (see  e.g.\ \cite[Theorem 1.26]{AG}), it follows that the optimal transport map satisfies $T=\nabla \varphi$, where $\varphi:\bar\Omega_1\to\R$ is a convex map. In addition, in virtue of  regularity results for the Monge-Amp\`ere equation (see \cite[Theorem 3.3]{DPHF}  and also \cite{Ca1,Ca2}), we know that $T=\nabla\varphi$ is a diffeomorphism of class  $C^{\infty}(\bar\Omega_1,\bar\Omega_2)$.
Hence we have that $\varphi$ is convex and of class $C^{\infty}$, and that $\nabla \varphi$ is a diffeomorphism onto its image. This implies that there exists $l>1$ such that for any $x\in \bar \Omega_1$ the eigenvalues of  the Hessian matrix of $\varphi$ at $x$, denoted by $\nabla^2\varphi(x)$,  are in the interval $(1/l, l)$. 
\\
\noindent
{\bf Step 1.}
We claim that there exist $O_1$ and $O_2$ bounded open sets with $\bar \Omega_1\subset O_1$ and $\bar \Omega_2\subset O_2$ and $\tilde T:O_1\to O_2$ such that $\tilde T_{|\bar \Omega_1}=T$ and $\tilde T$ is a diffeomorphism.  
To see that, let $\tilde \varphi: \R^{d}\to \R$ be a $C^{\infty}$ function satisfying $\tilde \varphi_{|\bar \Omega_1}=\varphi$ obtained by using Withney Extension Theorem (see \cite[Theorem~1]{W34}). 
Then, provided that  $O_1\supset \bar \Omega_{1}$ is chosen
small enough, for every $x\in O_1$ the eigenvalues of $\nabla^2\tilde \varphi (x)$ lie  in $(\frac{1}{l}, l)$.  
This implies that  $\nabla\tilde \varphi:O_1\to \R^d$ is a local diffeomorphism. 
Moreover, it is injective since it is the gradient of a strictly (actually strongly) convex function. 
Therefore, we conclude that $\nabla\tilde \varphi:O_1\to \nabla \tilde \varphi(O_1)=:O_2$ is a diffeomorphism, and 
we define $\tilde T:=\nabla \tilde \varphi$. \\
\noindent
{\bf Step 2.}
For every $t\in [0,1]$ let us introduce the map $\tilde T_t:O_1\to \R^d$ defined as
\begin{equation} \label{eq:interpolating_map}
    \tilde T_t:=(1-t)\Id+t \tilde T,
\end{equation}
where $\Id:\R^n\to \R^n$ is the identity function.
Then, we have that $\tilde T_t=\nabla \tilde \varphi_t$, where $\tilde \varphi_t:O_1\to\R$ is the strongly convex function satisfying $\tilde \varphi_t(x) :=  \frac{1-t}2|x|_2^2 + t\tilde \varphi(x)$ for every $t\in [0,1]$ and for every $x\in O_1$.
Using the same argument as before, we obtain that $\tilde T_t:O_1\to T_t(O_1) $ is a diffeomorphism onto its image. \\
\noindent
{\bf Step 3.}
Let us set the time-varying vector field $F$ as 
\begin{equation*}
F(t,y):=-\tilde T_t^{-1}(y)+\tilde T(\tilde T_t^{-1}(y)) \quad \mbox{for } (t,y)\in D,
\end{equation*}
where $D\subset [0,1]\times \R^n$ is the bounded set defined as
\begin{equation*}
D:=\{(t,y):t\in [0,1], y\in \tilde T_t(O_1)\}.
\end{equation*} 
Up to restricting $O_1$ if necessary, we have that $F\in C^{\infty}(\bar D, \R^n)$.  We finally take $\tilde F: [0,1]\times \R^n\to \R^n$, $C^{\infty}$ vector field  satisfying $\tilde F_{|\bar D}=F$ and with compact support. 
\\
{\bf Step 4.}
Let us denote with $\Psi:[0,1]\times \R^n\to \R^n$ the flow induced on $\R^n$ by the smooth and non-autonomous vector field $\tilde F$, i.e., 
\begin{equation}\label{eq:flow_proof_diffeot}
    \begin{cases}
        \frac{d}{dt} \Psi(t,x) = \tilde F(t, \Psi(t,x)) & t \in [0,1],\\
        \Psi(0,x) = x & x\in \R^n. 
    \end{cases}
\end{equation}
In order to conclude that the optimal transport map $T$ is isotopic to the identity, we need to show that, for every $x\in \bar \Omega_1$, we have $\Psi(1,x)=T(x)$.
To see that, we first observe that, from the definition \eqref{eq:interpolating_map}, it follows that $\tilde T_0 (x) = x$ for every $x\in \bar O_1$. Moreover, by differentiating in time \eqref{eq:interpolating_map}, we deduce that
\begin{equation*}
    \begin{split}
        \frac{d}{dt}\tilde T_t (x) &= -x + \tilde T(x)\\
        & =  - \tilde T_t^{-1}(\tilde T_t(x)) + \tilde T  (\tilde T_t^{-1} (\tilde T_t(x))) \\
        & = \tilde F(t, \tilde T_t(x) ).
    \end{split}
\end{equation*}
Therefore, combining the last computations with \eqref{eq:flow_proof_diffeot}, from the uniqueness of the solutions of ODEs we obtain that $\Psi(t,x) = \tilde T_t(x)$ for every $t\in[0,1]$ and for every $x\in O_1$. In particular, recalling that $\tilde T_1 (x) = \tilde T(x) = T(x) $ for every $x\in \bar \Omega_1$, we deduce that $T$ is isotopic to the identity.
\end{proof}

We report that the regularity hypothesis of Proposition~\ref{prop:isotopic} can be weakened by assuming that the densities are of class ${C}^{k}$ instead of ${C}^{\infty}$. In this case, the map $T$ is isotopic to the identity via a vector field of class $C^{k+1}$.

We state below the result concerning the approximation of the optimal transport map.

\begin{corollary} \label{cor:ot_map_approx}
    Under the same assumptions and notations as in Proposition~\ref{prop:isotopic}, let $T:\bar \Omega_1 \to \bar \Omega_2$ be the optimal transport map between $\mu$ and $\nu$. 
    Let $F_1,\ldots,F_k$ be a system of vector fields satisfying Assumption~\ref{ass:str_Lie}. 
    Then, for every $\e>0$, there exists an admissible control $u\in\U$ such that 
\begin{equation*}
\sup_{x\in \Omega_1}|T(x)-\Phi_u(x)|_2\leq \e,
\end{equation*}
where $\Phi_u$ is the flow corresponding to the control $u$ defined in \eqref{eq:def_Phi_u}.
\end{corollary}
\begin{proof}
    The proof follows immediately from Theorem~\ref{thm:univ_app_flows} and Proposition~\ref{prop:isotopic}.
\end{proof}
Corollary~\ref{cor:ot_map_approx} ensures that we can approximate the optimal transport map as the flow of a linear-control system. In general, we report that the problem of characterizing the functions that can be represented as flows of \emph{neural ODEs} (linear or non-linear in the controls) is an active field of research. For recent developments, we recommend \cite{kuehn_23}.  

In the next section we will study a functional whose minimization is related to the construction of the flow approximating the optimal transport map. Our approach is suitable for practical implementation, since, with a $\Gamma$-convergence argument, we can deal with the situation where only discrete approximations of $\mu, \nu$ are available.


\section{Optimal control problems and $\Gamma$-convergence} \label{sec:G-conv}

In this section we introduce a class of optimal control problems whose solutions play a crucial role in the construction of the approximating normalizing flows, and we establish a $\Gamma$-convergence result.
Here we work in a slightly more general framework than what is actually needed in the remainder of the paper.
For this reason, this part is divided into three subsections. In the first two, we present the existence and the $\Gamma$-convergence results for a broader class of problems, while in the last subsection we specialize to the problem of approximating the optimal transport map. In virtue of the $\Gamma$-convergence of the cost functionals, we can formulate a procedure of practical interest for the numerical approximation of the optimal transport map. 

Let $a:\R^n\times \R^n\to \R_+$ be a $C^1$-regular non-negative function,
and let $\gamma \in \mathcal{P}(\R^n\times\R^n)$ be a probability measure
with compact support. Namely, we assume that there exists a compact set $K\subset \R^n$ such that
$\mathrm{supp}(\gamma)\subset K\times K$.
For every $\beta>0$ we define the functional $\F^{\gamma,\beta}:\U\to\R_+$ as follows:
\begin{equation}\label{eq:def_functional_general}
    \F^{\gamma,\beta}(u):= \int_{\R^n\times\R^n}a(\Phi_u(x),y)\,d\gamma(x,y) + \frac{\beta}{2}\|u\|_{L^2}^2,
\end{equation}
where, for every $u\in\U$, the diffeomorphism
$\Phi_u:\R^n\to\R^n$ is the flow introduced 
in \eqref{eq:def_Phi_u}.
\subsection{Existence of minimizers} \label{subsec:exist_min}
Before proceeding we prove an auxiliary Lemma.

\begin{lemma} \label{lem:conv_intern_cost}
Let $a:\R^n\times \R^n\to \R_+$ be a $C^1$-regular 
non-negative function, and let $(u_m)_{m\in \NN}
\subset \U$ be a $L^2$-weakly convergent sequence,
i.e., $u_m\weak_{L^2}u_\infty$ as $m\to\infty$.
Finally, for every $m\in \NN\cup \{ \infty \}$,
let $\Phi_{u_m}:\R^n\to\R^n$ be the diffeomorphism
defined in \eqref{eq:def_Phi_u} and corresponding to the 
admissible control $u_m$.
Then, for every compact set 
$K'\subset \R^n\times \R^n$, we have that
\begin{equation} \label{eq:conv_intern_cost}
    \lim_{m\to\infty}\,
    \sup_{(x,y)\in K'}
    |a(\Phi_{u_m}(x),y)-a(\Phi_{u_\infty}(x),y)|=0.
\end{equation}
\end{lemma}
\begin{proof}
Since $K'\subset \R^n\times \R^n$ is compact, there
exist $K_1,K_2 \subset \R^n$ compact such that
$K\subset K_1\times K_2$.
Since the sequence $(u_m)_{m\in\NN}$ is weakly
convergent, there exists $\rho>0$ such that 
$\|u_m\|_{L^2}\leq \rho$ for every $m\in\NN\cup\{\infty\}$.
Therefore, in virtue of Lemma~\ref{lem:bound_Phi},
there exists a compact $\tilde K_1\subset \R^n$
such that
\begin{equation*}
    \Phi_{u_m}(K_1)\subset \tilde K_1
\end{equation*}
for every $m\in \NN\cup\{ \infty \}$. Since 
$a:\R^n\times \R^n\to\R_+$ is $C^1$-regular, we deduce that
the restriction $a\mid_{\tilde K_1 \times K_2}$ is
Lipschitz continuous with constant $\tilde L>0$,
which yields 
\begin{equation*}
    \sup_{(x,y)\in K_1\times K_2}
    |a(\Phi_{u_m}(x),y)-a(\Phi_{u_\infty}(x),y)|
    \leq \sup_{x\in K_1}\tilde L
    |\Phi_{u_m}(x)-\Phi_{u_\infty}(x)|_2
\end{equation*}
for every $m\in \NN$.
Then, owing to Proposition~\ref{prop:conv_flows}, from the previous inequality we deduce that
\begin{equation*}
    \lim_{m\to\infty} 
    \sup_{(x,y)\in K_1\times K_2}
    |a(\Phi_{u_m}(x),y)-a(\Phi_{u_\infty}(x),y)|
    =0.
\end{equation*}
Recalling that $K'\subset K_1\times K_2$ by construction, we have that \eqref{eq:conv_intern_cost} holds.
\end{proof}

In the next result we show that the functional $\F^{\gamma,\beta}$ defined in \eqref{eq:def_functional_general} admits a minimizer. Similarly as done in \cite{S_deep,S_ens}, the proof is based on the direct method of the Calculus of Variations.

\begin{proposition} \label{prop:exist_min_fun_gen}
Let $a:\R^n\times \R^n\to \R_+$ be a $C^1$-regular non-negative function,
and let $\gamma \in \mathcal{P}(\R^n\times\R^n)$ be a 
probability measure such that $\mathrm{supp}(\gamma)
\subset K\times K$, where $K\subset \R^n$ is 
a compact set. For every $\beta>0$, let 
$\F^{\gamma,\beta}:\U\to\R_+$ be the functional defined
in \eqref{eq:def_functional_general}. Then, there
exists $\hat u^{\gamma,\beta}\in \U$ such that
\begin{equation*}
    \F^{\gamma,\beta}(\hat u^{\gamma,\beta}) =
    \inf_{u\in\U}\F^{\gamma,\beta}(u).
\end{equation*}
\end{proposition}

\begin{proof}
Let us equip $\U$ with the weak topology of $L^2$.
In virtue of the direct method of Calculus of Variations
(see, e.g., \cite[Theorem~1.15]{D93}), it is sufficient to
prove that the functional $\F^{\gamma,\beta}$ is 
sequentially coercive and lower semi-continuous
with respect to the weak topology of $L^2$.
As regards the coercivity, we observe that
for every $u\in\U$ we have
\begin{equation*}
    \frac\beta2 \|u\|_{L^2}^2 \leq \F^{\gamma,\beta}(u), 
\end{equation*}
where we used the non-negativity of the function 
$a:\R^n\times\R^n\to\R_+$ associated to the 
integral cost in \eqref{eq:def_functional_general}.
The last inequality implies the inclusion
\begin{equation*}
    \{ u\in \U: \F^{\gamma,\beta}(u)\leq C \}
    \subset \left\{ u\in\U: \|u\|_{L^2}^2 \leq 
    2\frac{C}{\beta} \right\}
\end{equation*}
for every $C\geq 0$. This establishes the 
weak coercivity.
Let us consider a sequence of admissible controls 
$(u_m)_{m\in\NN}$ such that $u_m\weak_{L^2}u_\infty$ as
$m\to\infty$. We have to show that 
\begin{equation}\label{eq:lower_semic_fun}
    \F^{\gamma,\beta}(u_\infty)\leq \liminf_{m\to\infty}
    \F^{\gamma,\beta}(u_m).
\end{equation}
For every $m\in\NN\cup\{\infty\}$, let 
$\Phi_{u_m}:\R^n\to\R^n$ be the diffeomorphism
defined as in \eqref{eq:def_Phi_u} and corresponding to the
admissible control $u_m$.
Since the sequence $(u_m)_{m\in\NN}$ is weakly
convergent, there exists $\rho>0$ such that 
$\|u_m\|_{L^2}\leq \rho$ for every $m\in\NN\cup\{\infty\}$.
Therefore, we can apply Lemma~\ref{lem:conv_intern_cost}
to the compact set $K\times K\subset \R^n\times \R^n$ 
to deduce that
\begin{equation} \label{eq:cont_int_cost}
    \lim_{m\to\infty}
    \int_{\R^n\times \R^n}
    a(\Phi_{u_m}(x),y) \,d\gamma(x,y) =
    \int_{\R^n\times \R^n}
    a(\Phi_{u_\infty}(x),y) \,d\gamma(x,y),
\end{equation}
where we used 
the hypothesis $\mathrm{supp}(\gamma)\subset K\times K$.
In virtue of \eqref{eq:cont_int_cost}, we compute
\begin{align*}
    \liminf_{m\to\infty} \F^{\gamma,\beta}(u_m)&=
    \liminf_{m\to\infty}
    \left(
    \int_{\R^n\times \R^n}
    a(\Phi_{u_m}(x),y) \,d\gamma(x,y)
    + \frac\beta2 \|u_m\|_{L^2}^2
    \right)\\
    &=\int_{\R^n\times \R^n}
    a(\Phi_{u_\infty}(x),y) \,d\gamma(x,y)
    + \frac\beta2 \liminf_{m\to\infty}\|u_m\|^2_{L^2}.
\end{align*}
Recalling the lower semi-continuity of the $L^2$-norm with
respect to the weak convergence 
(see, e.g., \cite[Proposition~3.5]{B11}), 
the previous identity yields \eqref{eq:lower_semic_fun},
proving that $\F^{\gamma,\beta}$ is sequentially weakly
lower semi-continuous. This concludes the proof.
\end{proof}

\subsection{$\Gamma$-convergence result} \label{subsec:G_conv}
In Proposition~\ref{prop:exist_min_fun_gen} we have proved that the 
functional $\F^{\gamma,\beta}:\U\to\R_+$ \textcolor{red}{attains}
the minimum. We are now interested to study the 
stability of the problem of minimizing $\F^{\gamma,\beta}$ 
when the measure $\gamma\in \mathcal{P}(\R^n\times\R^n)$
is perturbed.

Let us consider a sequence $(\gamma_N)_{N\geq1}\subset
\mathcal{P}(\R^n\times \R^n)$ 
such that $\gamma_N\weak^* \gamma_\infty$ as $N\to\infty$
and such that there exists
a compact set $K\subset \R^n$ satisfying
$\mathrm{supp}(\gamma_N)\subset K\times K$
for every $N\geq 1$. We observe that from this assumptions
it follows that $\mathrm{supp}(\gamma_\infty)
\subset K\times K$ as well.
For every $N\in \NN\cup\{ \infty \}$
we define the functional
$\F^{N,\beta}:\U\to\R_+$ as follows:
\begin{equation}\label{eq:def_funct_N_Gamma}
    \F^{N,\beta}(u):=
    \int_{\R^n\times\R^n} a(\Phi_u(x),y) \,d\gamma_N(x,y)
    + \frac\beta2 \|u\|_{L^2}^2,
\end{equation}
where, for every $u\in\U$, $\Phi_u:\R^n\to\R^n$ is the 
flow defined as in \eqref{eq:def_Phi_u}.
The question that we are going to study is how the
minimizers of $\F^{\infty,\beta}$ relate to the
minimizers of $(\F^{N,\beta})_{N\geq 1}$.
We insist on the fact that the parameter $\beta>0$ 
is the same for all the functionals in consideration.
This fact is crucial to provide the following uniform bound
for the $L^2$-norm of the minimizers. 

\begin{lemma} \label{lem:unif_bound_norm_min}
Let $a:\R^n\times \R^n\to \R_+$ be a $C^1$-regular 
non-negative function,
and let $(\gamma_N)_{N\geq1}\subset
\mathcal{P}(\R^n\times \R^n)$  be a sequence of 
probability measures
such that $\gamma_N\weak^* \gamma_\infty$ as $N\to\infty$.
Let us further assume that there exists
a compact set $K\subset \R^n$ satisfying
$\mathrm{supp}(\gamma_N)\subset K\times K$
for every $N\in \NN\cup\{ \infty \}$. 
For every $N\in \NN\cup\{ \infty \}$, let 
$\F^{N,\beta}:\U\to\R_+$ be the functional
defined as in \eqref{eq:def_funct_N_Gamma}, and let 
$\hat u^{N,\beta}\in \U$ be any of its minimizers.
Then, there exists a constant $C>0$ such that
\begin{equation}\label{eq:uniform_bound_norm_min}
    \|\hat u^{N,\beta}\|_{L^2}^2 \leq  \frac{C}{\beta}.
\end{equation}
\end{lemma}

\begin{proof}
Let us consider the admissible control $\bar u \equiv 0$.
Then, observing that 
$\Phi_{\bar u}\equiv \mathrm{Id}_{\R^n}$,
we have that 
\begin{equation}\label{eq:unif_norm_1}
    \F^{N,\beta}(\bar u) = \int_{\R^n\times\R^n}
    a(x,y)\,d\gamma_N(x,y) \leq 
    \sup_{(x,y)\in K\times K} a(x,y)
\end{equation}
for every $N\in \NN\cup\{ \infty\}$.
On the other hand, if $\hat u^{N,\beta}\in \U$ is a 
minimizer of $\F^{N,\beta}$, we obtain
\begin{equation}\label{eq:unif_norm_2}
    \F^{N,\beta}(\bar u) \geq 
    \F^{N,\beta}(\hat u^{N,\beta}) \geq \frac\beta2
    \|\hat u^{N,\beta}\|_{L^2},
\end{equation}
where we used the non-negativity of the function $a$.
Finally, combining \eqref{eq:unif_norm_1} and 
\eqref{eq:unif_norm_2}, we deduce that
\eqref{eq:uniform_bound_norm_min} holds.
\end{proof}

We are now in position to establish a $\Gamma$-convergence
result for the sequence of functionals 
$(\F^{N,\beta})_{N\geq1}$. 
We recall below the definition of $\Gamma$-convergence. 
For a thorough discussion on this topic, we refer the 
reader to the textbook \cite{D93}.

\begin{defn}
\label{def:G_conv}
Let $(\X,d)$ be a metric space, and for every
$N\geq 1$ let $\G^N:\X\to\R\cup \{+\infty \}$
be a functional defined over $\X$. 
The sequence $(\G^N)_{N\geq 1}$ is said to 
$\Gamma$-converge to a functional 
$\G^\infty:\X\to\R\cup \{+\infty \}$ if the following
conditions are satisfied:
\begin{itemize}
\item \emph{liminf condition}: for every sequence
$(u_N)_{N\geq 1}\subset \X$ such that $u_N\to_\X u$
as $N\to\infty$ the following inequality holds
\begin{equation}\label{eq:liminf_cond}
\G^\infty(u) \leq \liminf_{N\to\infty} \G^N(u_N);
\end{equation}
\item \emph{limsup condition}: 
for every $u\in \X$ there exists a sequence 
$(u_N)_{N\geq 1}\subset \X$ such that 
$u_N\to_\X u$ as $N\to\infty$ and such that
the following inequality holds:
\begin{equation}\label{eq:limsup_cond}
\G^\infty(u) \geq \limsup_{N\to\infty} \G^N(u_N).
\end{equation}
\end{itemize}
If the conditions listed above are satisfied, then
we write $\G^N \to_\Gamma \G^\infty$ as $N\to\infty$.
\end{defn}

In calculus of variations $\Gamma$-convergence results are 
useful to relate the asymptotic behavior of the minimizers
of the converging functionals to the mimizers of the
$\Gamma$-limit. Indeed, if the elements 
of the $\Gamma$-convergent sequence $(\G^N)_{N\geq 1}$
are equi-coercive in the $(\X,d)$ topology, then  
if $\hat u_N\in \arg\min_\X \G^N$ for every $N\geq1$,
the sequence $(\hat u_N)_{N\geq1}$ is pre-compact
in $(\X,d)$ and any of its limiting point 
is a minimizer of $\G^\infty$ (see, e.g., 
\cite[Corollary~7.20]{D93}).

As done in the proof of
Proposition~\ref{prop:exist_min_fun_gen}, it is convenient
to equip the space of admissible controls $\U$ with the
weak topology of $L^2$. However, the weak topology
is metrizable only on bounded subsets of $\U$
(see \cite[Remark~3.3 and Theorem~3.29]{B11}).
Nevertheless, Lemma~\ref{lem:unif_bound_norm_min}
guarantees that the minimizers of $\F^{N,\beta}$ are
included in $\U_\beta$ for every $N\in\NN\cup\{ \infty \}$,
where we set
\begin{equation} \label{eq:def_U_beta}
    \U_\beta:= 
    \left\{ u\in \U: \|u\|_{L^2}^2\leq {C}/{\beta} 
    \right\},
\end{equation}
and $C$ is the constant prescribed by 
\eqref{eq:uniform_bound_norm_min}.
In other words, for every $N\in \NN\cup\{ \infty \}$
we can consider the restrictions
$\F^{N,\beta}|_{\U_\beta}:\U_\beta\to\R_+$
without losing any information
on the minimizers. With a slight abuse of notations,
we continue to use the symbol $\F^{N,\beta}$ to denote the
restricted functionals.
We are now in position to prove the main result of the 
present section.

\begin{theorem} \label{thm:G_conv}
Let $a:\R^n\times \R^n\to \R_+$ be a $C^1$-regular 
non-negative function,
and let $(\gamma_N)_{N\geq1}\subset
\mathcal{P}(\R^n\times \R^n)$  be a sequence of 
probability measures
such that $\gamma_N\weak^* \gamma_\infty$ as $N\to\infty$.
Let us further assume that there exists
a compact set $K\subset \R^n$ satisfying
$\mathrm{supp}(\gamma_N)\subset K\times K$
for every $N\in \NN\cup\{ \infty \}$.
For every $N\in \NN\cup\{ \infty \}$, let
$\F^{N,\beta}:\U_\beta\to\R_+$ be 
the functional defined as in \eqref{eq:def_funct_N_Gamma}
and restricted to the bounded subset $\U_\beta\subset\U$
introduced in \eqref{eq:def_U_beta}.
Then, if we equip $\U_\beta$ with the 
weak topology of $L^2$, 
we have that $\F^{N,\beta}\to_\Gamma
\F^{\infty,\beta}$ as $N\to\infty$.
\end{theorem}
\begin{proof}
We start by proving the \textit{liminf condition}.
Let $(u_N)_{N\geq1}\subset \U_\beta$ be a sequence such that
$u_N\weak_{L^2}u$ as $N\to\infty$. We have to prove that
\begin{equation}\label{eq:liminf_G_proof}
    \F^{\infty,\beta}(u) \leq \liminf_{N\to\infty}
    \F^{N,\beta}(u_N).
\end{equation}
Recalling that $\mathrm{supp}(\gamma_N)\subset K\times K
\subset \R^n\times \R^n$ for every
$N\in\NN\cup\{  \infty\}$, we  observe that 
\begin{align*}
    \int_{K\times K}a(\Phi_{u_N}(x),y)\,d\gamma_N(x,y)
    &= 
    \int_{K\times K}\big[a(\Phi_{u_N}(x),y) 
    - a(\Phi_u(x),y)
    \big]\,d\gamma_N(x,y)\\
    & \quad +  \int_{K\times K} 
     a(\Phi_u(x),y)\,d\gamma_N(x,y).
\end{align*}
In virtue of Lemma~\ref{lem:conv_intern_cost},
from the weak convergence $u_N\weak_{L^2}u$ 
as $N\to\infty$
we deduce that 
\begin{equation*}
    \lim_{N\to\infty} 
    \int_{K\times K}\big[a(\Phi_{u_N}(x),y) 
    - a(\Phi_u(x),y)
    \big]\,d\gamma_N(x,y) = 0.
\end{equation*}
Moreover, since by hypothesis $\gamma_N\weak^*\gamma_\infty$
as $N\to\infty$, we obtain that 
\begin{equation} \label{eq:conv_int_cost_Gamma}
    \lim_{N\to\infty}\int_{K\times K} 
    a(\Phi_{u_N}(x),y)\,d\gamma_N(x,y)
    = 
   \int_{K\times K} 
    a(\Phi_{u}(x),y)\,d\gamma_\infty(x,y).
\end{equation}
Finally, recalling that $u_N\weak_{L^2}u$ as $N\to\infty$
implies
\begin{equation*}
    \|u\|_{L^2} \leq \liminf_{N\to\infty}\|u_N\|_{L^2},
\end{equation*}
from \eqref{eq:conv_int_cost_Gamma} it follows that
\eqref{eq:liminf_G_proof} holds.\\
We now prove the \textit{limsup condition}. For every
$u\in \U_\beta$, let us set $u_N= u$ for every
$N\in \NN$. Then, using again the fact that
$\gamma_N\weak^*\gamma_\infty$ as $N\to\infty$, we have
\begin{equation*}
    \lim_{N\to\infty} \F^{N,\beta}(u) = 
    \lim_{N\to\infty} \int_{K\times K} a(\Phi_u(x),y)
    \, d\gamma_N(x,y) + \frac\beta2 \|u\|_{L^2}^2
    = \F^{\infty,\beta}(u).
\end{equation*}
This concludes the proof.
\end{proof}

As anticipated above, we can use the previous
$\Gamma$-convergence result to study the 
asymptotics of the minimizers of the functionals
$(\F^{N,\beta})_{N\geq1}$. 

\begin{corollary} \label{cor:conv_min}
Under the same assumptions as in Theorem~\ref{thm:G_conv},
we have that 
\begin{equation}\label{eq:conv_min}
    \lim_{N\to\infty}\min_{\U}\F^{N,\beta} =
    \min_{\U}\F^{\infty,\beta}.
\end{equation}
Moreover, if $\hat u_N\in \arg \min_\U \F^{N,\beta}$
for every $N\geq1$, then the sequence $(\hat u_N)_{N\geq1}$
is pre-compact with respect to the \emph{strong topology} of
$L^2$, and the limiting points are minimizers of 
the $\Gamma$-limit $\F^{\infty,\beta}$.
\end{corollary}

\begin{remark}
We insist on the fact that Corollary~\ref{cor:conv_min}
ensures that the sequence 
$(\hat u_N)_{N\geq1}$ is pre-compact with respect to the
\emph{strong topology} of $L^2$. Indeed, in general,
given a $\Gamma$-convergent sequence of equi-coercive
functionals, the standard theory guarantees that
any sequence of minimizers is pre-compact with respect
to the same topology used to establish the 
$\Gamma$-convergence (see \cite[Corollary~7.20]{D93}). 
Thus, in our case, this fact would immediately imply that
$(\hat u_N)_{N\geq1}$ is pre-compact with respect to
the \emph{weak} topology of $L^2$. However, 
in the case of the functionals considered here, we 
can strenghten this fact and we can deduce the 
pre-compactness also in the strong topology.
We report that similar phenomena have been described 
in \cite{S_deep,S_ens}.
\end{remark}

\begin{proof}[Proof of Corollary~\ref{cor:conv_min}]
Owing to Lemma~\ref{lem:unif_bound_norm_min}, we have that
\begin{equation} \label{eq:min_restr_min_full}
    \min_{\U}\F^{N,\beta} = \min_{\U_\beta}\F^{N,\beta}
\end{equation}
for every $N\in \NN\cup\{ \infty \}$.
Moreover, since the restricted functionals 
$\F^{N,\beta}:\U_\beta\to\R_+$ are $\Gamma$-convergent in
virtue of Theorem~\ref{thm:G_conv}, from 
\cite[Corollary~7.20]{D93} we obtain that 
\begin{equation} \label{eq:conv_min_restr}
    \lim_{N\to\infty}\min_{\U_\beta}\F^{N,\beta} =
    \min_{\U_\beta}\F^{\infty,\beta}.
\end{equation}
Combining \eqref{eq:conv_min_restr} and
\eqref{eq:min_restr_min_full}, we deduce 
\eqref{eq:conv_min}.
As regards the pre-compactness of the minimizers, let 
us consider a sequence $(\hat u_N)_{N\geq1}$ 
such that $\hat u_N\in \arg \min_{\U} \F^{N,\beta}$
for every $N\geq1$.
Using again \cite[Corollary~7.20]{D93}, it follows that
$(\hat u_N)_{N\geq1}$ is pre-compact with respect to the
weak topology of $L^2$, and that its limiting points
are minimizers of $\F^{\infty,\beta}$.
Let $(\hat u_{N_m})_{m\geq1}$ be
a sub-sequence such that 
$\hat u_{N_m}\weak_{L^2}\hat u_\infty$ as
$m\to\infty$. On one hand, using \eqref{eq:conv_min}
we have that
\begin{equation}\label{eq:limit_min_subseq}
    \lim_{m\to\infty}\F^{N_m,\beta}(\hat u_{N_m})
    = \F^{\infty,\beta}(\hat u_\infty).
\end{equation}
On the other hand, the same argument used to establish 
\eqref{eq:conv_int_cost_Gamma} yields
\begin{equation} \label{eq:conv_int_cost_minimz}
    \lim_{m\to\infty}\int_{K\times K} 
    a(\Phi_{\hat u_{N_m}}(x),y)\,d\gamma_{N_m}(x,y)
    = 
   \int_{K\times K} 
    a(\Phi_{\hat u_{\infty}}(x),y)\,d\gamma_\infty(x,y).
\end{equation}
Therefore, combining \eqref{eq:limit_min_subseq}- \eqref{eq:conv_int_cost_minimz} and recalling
the expression of $\F^{N,\beta}$ in 
\eqref{eq:def_funct_N_Gamma}, we deduce that
\begin{equation*}
    \lim_{m\to\infty}\|\hat u_{N_m}\|_{L^2}
    = \|\hat u_\infty\|_{L^2}.
\end{equation*}
\end{proof}

\subsection{Optimal transport map approximation} \label{subsec:OT_map_appr}
In this subsection we will discuss how the $\Gamma$-convergence result established in the previous part can be exploited for the problem of the optimal transport map approximation.
In this setting, the measures $(\gamma_N)_{N\geq 1}$ are chosen in a specific way.
Indeed, given two probability measures $\mu,\nu \in \Prob(\R^n)$ with supports included in the compact set $K\subset \R^n$, we consider two sequences $(\mu_N)_{N\geq 1}, (\nu_N)_{N\geq 1} \subset \Prob(K)$ such that $\mu_N \weak^* \mu$ and $\nu_N \weak^* \nu$ as $N\to\infty$. 
Moreover, in this part, for every $N\geq 1$ we choose $\gamma_N\in \opt(\mu_N,\nu_N)$, i.e., an optimal transport plan between $\mu_N$ and $\nu_N$ with respect to the Euclidean squared distance (see the definition in \eqref{eq:def_opt_plans}).
In view of practical applications, $\mu_N$ and $\nu_N$ can be thought as discrete (or empirical) approximations of the measures $\mu$  and $\nu$, respectively. Finally, here we set the cost function $a:\R^n\times\R^n\to \R_+$ to be $a(x,y):=|x-y|^2_2$, so that the functionals $\F^{N,\beta}:\U\to \R_+$ have the form
\begin{equation} \label{eq:funct_N_OT_appr}
    \F^{N,\beta}(u) = \int_{\R^n\times\R^n} |\Phi_u(x) - y|^2_2 \, d\gamma_N(x,y) + \frac{\beta}{2} \| u \|_{L^2}^2,
\end{equation}
while the set $\U_\beta$ is defined as in Subsection~\ref{subsec:G_conv} (see \eqref{eq:def_U_beta}).
We are now in position to state the result that motivated this paper.

\begin{theorem}\label{thm:OT_map_appr}
    Let $\mu,\nu \in \Prob(\R^n)$ be two probability measures with supports included in the compact set $K\subset \R^n$, and such that $\mu \ll \mathcal{L}^n$, and let us consider $(\mu_N)_{N\geq 1}, (\nu_N)_{N\geq 1} \subset \Prob(K)$ such that $\mu_N \weak^* \mu$ and $\nu_N \weak^* \nu$ as $N\to\infty$. Let us consider $(\gamma_N)_{N\geq 1}$ such that $\gamma_N\in \opt(\mu_N,\nu_N)$ for every $N\geq 1$.
    Let $\F^{N,\beta}:\U_\beta\to\R_+$ be the functional defined as in \eqref{eq:funct_N_OT_appr} and restricted to the bounded subset $\U_\beta\subset\U$ introduced in \eqref{eq:def_U_beta}.
    Then, if we equip $\U_\beta$ with the weak topology of $L^2$, we have that $\F^{N,\beta}\to_\Gamma\F^{\infty,\beta}$ as $N\to\infty$, where
    \begin{equation}\label{eq:funct_OT_app}
        \F^{\infty,\beta}(u) = \int_{\R^n} |\Phi_u(x)-T(x)|_2^2 \,d\mu(x) + \frac{\beta}{2} \| u \|_{L^2}^2,
    \end{equation}
    and $T:\supp(\mu)\to \supp(\nu)$ is the optimal transport map between $\mu$ and $\nu$ with respect to the Euclidean squared distance.
    Moreover, we have that 
    \begin{equation*}
        \lim_{N\to\infty}\min_{\U}\F^{N,\beta} =
       \min_{\U}\F^{\infty,\beta},
    \end{equation*}
    and, if $\hat u_N\in \arg \min_\U \F^{N,\beta}$ for every $N\geq1$, then the sequence $(\hat u_N)_{N\geq1}$ is pre-compact with respect to the strong topology of $L^2$, and the limiting points are minimizers of the $\Gamma$-limit $\F^{\infty,\beta}$.
\end{theorem}

\begin{proof}
    From Proposition~\ref{prop:conv_opt_plans} it follows that the sequence $(\gamma_N)_{N\geq 1}$ is pre-compact and that the limiting points are included in $\opt(\mu,\nu)$. Since $\mu\ll \mathcal{L}^n$, from Brenier's Theorem (see, e.g., \cite[Theorem~2.26]{AG}) we deduce that $\opt(\mu,\nu) = \{(\Id,T)_\# \mu\}$, where $T:\supp(\mu)\to \supp(\nu)$ is the optimal transport map between $\mu$ and $\nu$. Therefore, we have that $\gamma_N \weak^* \gamma_\infty$ as $N\to\infty$, where we set $\gamma_\infty:=(\Id,T)_\# \mu$. 
    Then, the theses are a direct consequence of Theorem~\ref{thm:G_conv} and of Corollary~\ref{cor:conv_min}.
\end{proof}

{
\begin{remark} \label{rmk:nonopt_plans}
We observe that the conclusion of the previous result holds as well even when the coupling $\gamma_N$ has not been obtained by solving the discrete optimal transport problem between $\mu_N$ and $\nu_N$. 
Namely, as soon as $\gamma_N\weak^*\gamma=(\mathrm{Id},T')_\#\mu$ as $N\to\infty$ for a measurable transport map $T':\R^d\to\R^d$, the $\Gamma$-convergence result holds, after substituting $T'$ to $T$ in \eqref{eq:funct_OT_app}.
Nevertheless, in view of applications, thinking $\gamma_N$ as an (approximate) optimal coupling looks particularly convenient, since we can take advantage of well-established and efficient computational methods (see e.g. \cite{Cut13, PeyCut}).
Moreover, in the case of a generic transport map $T'$, we lack an approximation result analogous to Corollary~\ref{cor:ot_map_approx}, unless $T'$ is not in turn a diffeomorphism isotopic to the identity.
\end{remark}
}

\begin{remark}\label{rmk:app_OT_map_L2}
    We observe that, under the same assumptions as in Corollary~\ref{cor:ot_map_approx}, for every $\varepsilon>0$, there exists $\bar \beta>0$ such that, for every $\beta\in (0,\bar \beta]$, we have $\kappa(\beta)\leq \varepsilon$, where $\kappa:[0,+\infty)\to [0,+\infty)$ is defined as
    \begin{equation}\label{eq:def_kappa}
        \kappa(\beta):= \sup \left\{ \int_{\R^n} |\Phi_u(x)-T(x)|_2^2 \,d\mu(x) : u\in \arg\min \F^{\infty,\beta} \right\}.
    \end{equation}
    Indeed, given $\varepsilon>0$, in virtue of Corollary~\ref{cor:ot_map_approx}, there exists a control $\tilde u \in \U$ such that
    \begin{equation*}
        \sup_{x\in K} |\Phi_{\tilde u}(x)- T(x)|_2^2 \leq \frac\e2.
    \end{equation*}
    Moreover, if we choose $\bar \beta>0$ such that $\bar \beta \|\tilde u\|_{L^2}^2 = \e $, then, for every $\beta\in (0,\bar \beta ]$, we obtain $\F^{\infty,\beta}(\tilde u) \leq \e$. 
    Being $\tilde u\in \U$ a competitor for the minimization of $\F^{\infty,\beta}$, we deduce that $\kappa(\beta)\leq \varepsilon$  for every $\beta\in (0,\bar \beta]$. We report that this argument has already been used in \cite[Proposition~5.4]{S_deep}.
    This observation guarantees that, by tuning the parameter $\beta>0$ to be small enough, if $\hat u_\beta \in \arg\min \F^{\infty,\beta}$, then the corresponding flow $\Phi_{\hat u_\beta}$ provides an approximation of the optimal transport map $T:\supp(\mu)\to \supp(\nu)$ which is arbitrarily accurate in the $L^2_\mu$-strong topology.
    The interesting aspect is that an approximation of $T$ can be carried out by minimizing a functional over the Hilbert space $\U$ of the admissible controls.
    Even though handling $\F^{\infty,\beta}$ already requires the knowledge of the optimal transport map $T$, the $\Gamma$-convergence result ensures that we can construct the approximation by minimizing the functionals $\F^{N,\beta}$ instead of  $\F^{\infty,\beta}$.  
    In Remark~\ref{rmk:triang_ineq_app_OT} we discuss in detail the more applicable situation when dealing with discrete approximations $\mu_N,\nu_N$ of $\mu,\nu$, respectively.
    Finally, we stress the fact that, in general, this approach does not provide a reconstruction of the optimal transport map that is close also in the $C^0$-norm.
\end{remark}

\begin{remark}\label{rmk:triang_ineq_app_OT}
    In view of a possible practical implementation, we recall that we aim at producing a flow $\Phi_u:\R^n\to \R^n$ with a suitable control $u\in \U$ such that the distance $W_2 (\Phi_{u \, \# }\mu, \nu)$ is as small as desired, where $\mu, \nu$ are probability measures satisfying the same assumptions as in Corollary~\ref{cor:ot_map_approx}.
    Here it is important to stress that $\mu$ and $\nu$ do not play a symmetric role in the applications: indeed, it is convenient to understand $\mu$ as a known object (i.e., whose density is known, or which it is inexpensive to sample from), while $\nu$ denotes a probability measure which we have limited information about, and it is complicated (but not impossible) to gather new samplings.
    In this framework, we imagine that we have at our disposal discrete approximations $\mu_N, \nu_N$ of $\mu, \nu$, respectively. We provide below an asymptotic estimate of $W_2 (\Phi_{u \, \# }\mu, \nu)$ for large $N$ when $u$ is obtained by minimizing the functional $\F^{N,\beta}$ defined in \eqref{eq:funct_N_OT_appr}.
    Namely, if we take $\hat u_{N, \beta} \in \arg \min_\U \F^{N, \beta}$, when $N\gg 1$ we have
    \begin{equation} \label{eq:asympt_err}
         W_2(\Phi_{\hat u_{N,\beta} \, \# }\mu, \nu) \leq
          L_\beta W_2(\mu, \mu_N) +
          2\sqrt{\kappa(\beta)}
          +  W_2(\nu_N, \nu),
    \end{equation}
    where $L_\beta \to +\infty$ and $\kappa(\beta)\to 0$ as $\beta \to 0$.
    To see that, using the triangular inequality, we compute for any $u\in \mathcal{U}$
    \begin{equation} \label{eq:tri_ineq_W}
        W_2(\Phi_{u \, \# }\mu, \nu) \leq 
        L_{\Phi_{u}} W_2(\mu, \mu_N) + 
        W_2 (\Phi_{u \, \# }\mu_N, \nu_N) +
        W_2(\nu_N, \nu),
    \end{equation}
    where $L_{\Phi_{u}}$ denotes the Lipschitz constant of the flow $\Phi_u$.
    In addition, if $\gamma_N \in \opt(\mu_N, \nu_N)$, we observe that
    \begin{equation*}
        W_2^2(\Phi_{u \, \# }\mu_N, \nu_N) 
        \leq \int_{\R^n\times \R^n} |\Phi_u(x) - y|_2^2 \, d\gamma_N(x,y),
    \end{equation*}
    where we used the fact that $(\Phi_u, \Id)_\#\gamma_N\in \mathrm{Adm}(\Phi_{u \, \#}\mu_N, \nu_N)$.
    For every $N\geq 1$, let us finally consider $\hat u_{N, \beta} \in \arg \min_\U \F^{N, \beta}$.
    Using the same computations as in \eqref{eq:conv_int_cost_minimz}, it turns out that
    \begin{equation*}
        \limsup_{N\to \infty} \int_{\R^n\times \R^n} |\Phi_{\hat u_{N,\beta}}(x) - y|_2^2 \, d\gamma_N(x,y)
        \leq
        \kappa(\beta),
    \end{equation*}
    where $\kappa:[0,+\infty)\to [0,+\infty)$ is the application defined in \eqref{eq:def_kappa}.
    Combining the last two inequalities, we deduce that
    \begin{equation} \label{eq:asympt_est_training}
        \limsup_{N\to\infty} W_2(\Phi_{\hat u_{N,\beta} \, \# }\mu_N, \nu_N) \leq \sqrt{\kappa(\beta)}.
    \end{equation}
    Moreover, since Lemma~\ref{lem:unif_bound_norm_min} guarantees that $\| \hat u_{N,\beta} \|_{L^2} \leq C/\beta$ for every $N\geq 1$, it follows from Lemma~\ref{lem:Lip_flow} that there exists a constant $L_\beta>0$ independent on $N$ such that $L_{\Phi_{\hat u_{N,\beta}}}\leq L_\beta$.
    Using this consideration and \eqref{eq:asympt_est_training}, from \eqref{eq:tri_ineq_W} we obtain the asymptotic estimate \eqref{eq:asympt_err}. We recall that in \eqref{eq:asympt_err} $L_\beta \to +\infty$ and $\kappa(\beta)\to 0$ as $\beta \to 0$. The constant $L_\beta$ may be large for $\beta$ close to $0$, however this is mitigated by the fact that $W_2(\mu,\mu_N)$ can be made small at a reasonable cost.
\end{remark}

\begin{remark}\label{rmk:geodesics}
    For every $u\in \U$, let $\Phi_u^{(0,t)}:\R^n\to\R^n$ be the flow induced by evolving the linear-control system \eqref{eq:lin_ctrl} in the time interval $[0,t]$, for every $t\leq 1$. If, for a given $u\in \U$, the final-time flow $\Phi_u=\Phi_u^{(0,1)}$ provides an approximation of the optimal transport map $T$ between $\mu$ and $\nu$ with respect to the squared Euclidean distance, a natural question is whether the curve $t\mapsto \Phi_{u\, \#}^{(0,t)} \mu$ is close to the Wasserstein $W_2$-geodesic that connects $\mu$ to $\nu$.
    In general, the answer is negative. However, it is possible to construct an approximation of the Wasserstein geodesic using the final-time flow $\Phi_u$.
    Indeed, the $W_2$-geodesic connecting $\mu$ to $\nu$ has the form $t\mapsto \eta_t:= ((1-t)\Id + t T)_\# \mu$ (see, e.g., \cite[Remark~3.13]{AG}). Similarly, exploiting the fact that $\Phi_u$ is close to $T$, we can define the curve $t\mapsto\tilde \eta_t:=((1-t)\Id + t \Phi_u)_\# \mu$, and we can compute
    \begin{equation*}
        \begin{split}
            W_2^2(\eta_t,\tilde \eta_t) &=
            W^2_2\big(((1-t)\Id + t T)_\# \mu, ((1-t)\Id + t \Phi_u)_\# \mu\big)\\
            &\leq t^2 \int_{\R^n} |\Phi_u(x)- T(x)|_2^2\, d\mu(x) = t^2 \| \Phi_u - T \|_{L^2_\mu}^2,
        \end{split}
    \end{equation*}
    i.e., we can estimate instant-by-instant the deviation of $\tilde \eta$ from the geodesic connecting $\mu$ to $\nu$ in terms of the $L^2_\mu$ distance between $T$ and $\Phi_u$.
    This is relevant, since the latter is precisely the integral term involved in the functional \eqref{eq:funct_OT_app}.
\end{remark}


\section{Numerical approximation of the optimal transport map}
\label{sec:num_approx}
In this section, we propose a numerical approach for the construction of a \textit{normalizing flow} $\Phi_u:\R^n\to\R^n$ generated by a linear-control system, such that the push-forward $\Phi_{u\, \#}\mu$ is close to $\nu$ in  the $W_2$-distance, where $\mu,\nu$ are two assigned probability measures on $\R^n$.
In order to consider a more realistic framework, we deal with $\mu_N, \nu_N$, that represent discrete probability measures with small $W_2$-distance to $\mu,\nu$, respectively.
On one hand, under the assumption that the measure $\mu$ is known, the construction of $\mu_N$ can be customized by the user. 
In general, the problem of approximating a probability measure with a convex combination of a fixed number of Dirac deltas is currently an active topic of research (see, e.g., \cite{MSS21}).
On the other hand, the measure $\nu_N$ should be thought as assigned.
After the preliminary computation of an optimal transport plan between $\mu_N$ and $\nu_N$ with respect to the Euclidean squared norm, we shall write an optimal control problem, and we address its numerical resolution with an iterative method originally proposed in \cite{SS80} and based on the Pontryagin Maximum Principle.


\subsection{Preliminary optimal transport problem} \label{subsec:prel_OT_problem}
The first step for the construction of the functional $\F^{N,\beta}:\U\to\R$ defined as in \eqref{eq:funct_N_OT_appr} is the computation of an optimal transport plan $\gamma_N\in\opt(\mu_N,\nu_N)$. In this case, for every $u\in \U$ the functional $\F^{N,\beta}$ can be rewritten as follows:
\begin{equation} \label{eq:funct_real_problems}
    \F^{N,\beta}(u) = \sum_{\substack{i=1,\ldots,N_1 \\ j=1,\ldots,N_2}} \gamma_N^{i,j}|\Phi_u(x_i)-y_j|_2^2
    + \frac\beta2 \| u \|_{L^2}^2,
\end{equation}
where $\supp(\mu_N)=\{ x_1,\ldots,x_{N_1} \}$, $\supp(\nu_N)=\{ y_1,\ldots,y_{N_2} \}$, and $\gamma_N = (\gamma_N^{i,j})^{j=1,\ldots,N_2}_{i=1,\ldots,N_1  }$ is the optimal transport plan.
It is well-known (see  \cite[Proposition~3.4]{PeyCut} and \cite[Theorem~8.1.2]{Brualdi}) that, if $\#\supp(\mu_N)=N_1$ and $\#\supp(\nu_N)=N_2$, then, there exists at least an optimal transport plan $\gamma_N\in\opt(\mu_N,\nu_N)$ such that $\#\supp(\gamma_N)\leq N_1 + N_2$ (see also \cite{AV22} for further details).
In our case, having a \textit{sparse} optimal transport plan (i.e. $\#\supp(\gamma_N)\ll N_1N_2$) is useful to alleviate the computations, since this reduces the number of terms that appear in the sum in \eqref{eq:funct_real_problems}. 
In order to achieve that while computing numerically $\gamma_N = (\gamma_N^{i,j})^{j=1,\ldots,N_2}_{i=1,\ldots,N_1}$, it could be appropriate to introduce a \textit{quadratic regularization}
(see, e.g., \cite{BSR18,LMM21}).


\subsection{Pontryagin Maximum Principle}\label{subsec:PMP}
In this subsection we formulate the necessary optimality conditions for the minimization of the functional $\F^{N,\beta}$ defined in \eqref{eq:funct_real_problems}.
We observe that this minimization can be naturally formulated as an optimal control problem in $(\R^n)^{N_1}$, where $N_1\geq1$ stands for the number of atoms $\{x_1,\ldots,x_{N_1}\}$ that constitute the probability measure $\mu_N$.
More precisely, if we denote by $Z=(z_1,\ldots,z_{N_1})$ a point in $(\R^n)^{N_1}$, the control system that we consider has the form
\begin{equation}\label{eq:ext_ctrl_PMP}
    \begin{cases}
        \dot z_i(t) = F(z_i(t))u(t) &\mbox{a.e. in }[0,1],\\
        z_i(0) = x_i,
    \end{cases}
    \quad \mbox{for } i=1,\ldots,N_1,
\end{equation}
where the function $F:\R^n\to\R^{n\times k}$ is the same that prescribes the dynamics in \eqref{eq:lin_ctrl}. We use the notation $Z^u:[0,1]\to (\R^n)^{N_1}$ to indicate the solution of \eqref{eq:ext_ctrl_PMP} corresponding to the admissible control $u\in \U$.
We insist on the fact that the components $z_1,\ldots,z_{N_1}$ are \emph{simultaneously driven} by the control $u\in \U$. Finally, the function associated to the terminal cost (i.e., the first term at the right-hand side of \eqref{eq:funct_real_problems}) is
\begin{equation*}
    Z=(z_1,\ldots,z_{N_1})\mapsto \sum_{\substack{i=1,\ldots,N_1 \\ j=1,\ldots,N_2}} \gamma_N^{i,j}|z_i-y_j|_2^2.
\end{equation*}
We state below the Maximum Principle for our particular optimal control problem. For a detailed and general presentation of the topic the reader is referred to the textbook \cite[Chapter~12]{AS04}. 

\begin{theorem}\label{thm:PMP}
    Let $\hat u\in \U$ be an admissible control that minimizes the functional $\F^{N,\beta}$ defined in \eqref{eq:funct_real_problems}. 
    Let $\mathcal{H}:(\R^n)^{N_1}\times ((\R^n)^{N_1})^* \times\R^k\to\R$
    be the hamiltonian function defined as follows:
    \begin{equation}\label{eq:ham_PMP}
       \mathcal{H}(Z,\Lambda,u) =\sum_{i=1}^{N_1}
        \lambda_{i}\cdot  F(z_i)u 
        - \frac\beta2|u|^2,
    \end{equation}
    where we set $Z=(z_{1},\ldots,z_{N_1})$ and $\Lambda=(\lambda_{1},\ldots,\lambda_{N_1})$, with $\lambda_i \in (\R^n)^*$. 
    Then there exists an absolutely continuous function $\Lambda^{\hat u}:[0,1]\to(\R^n)^{N_1}$ such that the following conditions hold:
    \begin{itemize}
        \item For every $i=1,\ldots,N_1$  the curve $z^{\hat u}_{i}:[0,1]\to\R^n$ satisfies \begin{equation}\label{eq:ode_traj_PMP}
            \begin{cases}
                \dot z^{\hat u}_{i}(t) =  \frac{\partial
                }{\partial \lambda_{i}}\mathcal{H}(Z^{\hat u}(t),
                \Lambda^{\hat u}(t),\hat u(t))
                &\mbox{a.e. in }[0,1],\\
                z^{\hat u}_{i}(0) = x_i;
            \end{cases}
        \end{equation}
        \item For every $i=1,\ldots,N_1$  the curve $\lambda^{\hat u}_{i}:[0,1]\to (\R^n)^*$ satisfies
        \begin{equation}\label{eq:ode_covec_PMP}
            \begin{cases}
                \dot \lambda^{\hat u}_i \textcolor{red}{(t)} = - \frac{\partial
                }{\partial \textcolor{red}{z_{i}}}\mathcal{H}(Z^{\hat u}(t),
                \Lambda^{\hat u}(t),\hat u(t))
                &\mbox{a.e. in }[0,1],\\
                \lambda^{\hat u}_{i}(1) = - \sum_{ j=1,\ldots,N_2} \gamma_N^{i,j}(z_i^{\hat u}(1)-y_j);
            \end{cases}
        \end{equation}
        \item For a.e. $t\in[0,1]$, the following condition is satisfied:
        \begin{equation}\label{eq:max_cond_PMP}
            \hat u(t) \in \arg \max_{u\in \R^k}
            \mathcal{H}(Z^{\hat u}(t), \Lambda^{\hat u}(t),u).
        \end{equation}
    \end{itemize}
\end{theorem}

\begin{remark}
    In Theorem~\ref{thm:PMP} we stated the Pontryagin Maximum Principle for normal extremals only. This is due to the fact that the optimal control problem concerning the minimization of \textcolor{red}{$\F^{N,\beta}$} does not admit abnormal extremals.
\end{remark}


\subsection{Algorithm description}

In this subsection we describe the implementable algorithm that we employed to carry out the numerical simulation described in the next section. We address the numerical minimization of the functional $\F^{N,\beta}$ introduced in \eqref{eq:funct_real_problems} using the iterative method proposed in \cite{SS80}, based on the Pontryagin Maximum Principle.
This approach has been recently applied in \cite{S_deep,S_ens} for the task of recovering a diffeomorphism from observations, and for the simultaneous optimal control of an ensemble of systems, respectively.

Before proceeding, we describe the discretization of the dynamics \eqref{eq:ext_ctrl_PMP} and how we reduce the minimization of \eqref{eq:funct_real_problems} to a finite dimensional problem. 
Let us consider the evolution time horizon $[0,1]$, and for $M\geq 2$ let us take the equispaced nodes $\{ 0, \frac1M,\ldots, \frac{M-1}M, 1 \}$. Recalling that $\U:=L^2([0,1],\R^k)$, we define the subspace $\U_M\subset \U$ as follows:
\begin{equation*} 
u\in \U_M \iff u(t) =
\begin{cases}
u_1& \mbox{if } 0\leq t< \frac1M\\
\vdots\\
u_M & \mbox{if } \frac{M-1}{M}\leq t\leq 1,
\end{cases}
\end{equation*}
where $u_1,\ldots,u_M\in \R^k$.
For every $l=1,\ldots,M$, we shall write $u_l = (u_{1,l},\ldots,u_{k,l})$ to denote the components of $u_l\in\R^k$. Then, any element $u\in\U_M$ will be represented by the following array:
\begin{equation*}
u = (u_{j,l})^{j=1,\ldots,k}_{l=1,\ldots,M}.
\end{equation*}
For every $i=1,\ldots,N_1$, let $z^u_{i}:[0,1]\to \R^n$ be the solution of \eqref{eq:ext_ctrl_PMP} corresponding to the $i$-th athom of the measure $\mu_N$ and to the control $u$.
Then, for every $i=1,\ldots,N_1$ and $l=0,\ldots,M$, we define the array that collects the evaluation of the trajectories at the time nodes:
\begin{equation*}
(z^l_i)_{i=1,\ldots,N_1}^{l=0,\ldots,M}, \qquad
z^l_i := z^u_{i}\left({l}/{M}\right) \in \R^n,
\end{equation*}
where we dropped the reference to the control that generates the trajectories. This is done to avoid hard notations, since we hope that it will be clear from the context the correspondence between trajectories and control.
For the approximate resolution of the \emph{forward dynamics} \eqref{eq:ext_ctrl_PMP} we use the explicit Euler scheme, i.e.,
\begin{equation*}
    z_i^0 = x_i, \qquad z_i^{l+1} = z_i^l + \frac1M F(z_i^l)u_l
\end{equation*}
for $i=1,\ldots,N_1, \ l=0,\ldots,M-1$.
Similarly, for every $i=1,\ldots,N_1$, let $\lambda^u_i:[0,1]\to(\R^n)^*$ be the solution of \eqref{eq:ode_covec_PMP} corresponding to the control $u$, and let us introduce the corresponding array of the evaluations:
\begin{equation*}
(\lambda^l_i)_{i=1,\ldots,N_1}^{l=0,\ldots,M},\qquad
\lambda^l_i:= \lambda^u_i\left( {l}/{M} \right)\in (\R^n)^*,
\end{equation*}
and we approximate the \emph{backward dynamics} \eqref{eq:ode_covec_PMP} with the implicit Euler scheme:
\begin{equation*}
    \lambda_i^M = - \sum_{ j=1,\ldots,N_2} \gamma_N^{i,j}(z_i^M-y_j), 
    \qquad  \lambda_i^{l-1} = \lambda_i^l + \frac1M \left( \lambda_i^{l-1} \cdot  \frac{\partial}{\partial z}F(z_i^{l-1})u_l \right)
\end{equation*}
for $i=1,\ldots,N_1, \ l=M,\ldots,1$.\\
The method is described in Algorithm~\ref{alg:iter_PMP}.

\begin{algorithm}
\scriptsize
\KwData{ 
\begin{itemize}
\item $F:\R^n\to \R^{n\times k}$ controlled fields;
\item $(x_i)_{i=1,\ldots,N_1}$ atoms of $\mu_N$;
\item $(y_i)_{i=1,\ldots,N_2}$ atoms of $\nu_N$;
\item $\gamma_N = 
(\gamma_N^{i,j})^{j=1,\ldots,N_2}_{i=1,\ldots,N_1}
\in \opt(\mu_N,\nu_N)$.
\end{itemize}
{\bf Algorithm setting:} 
$M = $ n. sub-intervals of $[0,1]$, $h=\frac1M$, 
$0<\tau<1$, 
$\rho >0$, $\max_{\mathrm{iter}}\geq 1$}

Initial guess for $u\in \U_M$\;

\For(\tcp*[f]{First computation of  
trajectories}){$i=1,\ldots,N_1$ }{
		Compute $(z^l_i)^{l=1,\ldots,M}$ using
		$(u_l)_{l=1,\ldots,M}$ and $x_i$\;
	}

$\mathrm{Cost}\gets \sum_{i=1,\ldots,N_1}^{j=1,\ldots,N_2} \gamma_N^{i,j}|z^M_i-y_j|_2^2
    + \frac\beta2 \| u \|_{L^2}^2$\;
$\mathrm{flag}\gets 1$\;
\For(\tcp*[f]{Iterations of Iterative Maximum Principle}){$r=1,\ldots,\max_{\mathrm{iter}}$ }{
	\If(\tcp*[f]{Update covectors only if necessary}){$\mathrm{flag}=1$ }{
		\For(\tcp*[f]{Backward computation of 
		covectors }){$i=1,\ldots,N_1$ }{
			$\lambda^M_i\gets - \sum_{ j=1}^{N_2} \gamma_N^{i,j}(z_i^M-y_j)
			$\;
			Compute $(\lambda^l_i)^{l=0,\ldots,M-1}$
			using $(u_l)_{l=1,\ldots,M}$, 
			$\textcolor{red}{(z_i^l)}^{l=0,\ldots,M}$ and 
			$\lambda^M_i$\;					
			}
		}
		
	$(z_i^{0,\mathrm{new}})^{i=1,\ldots,N_1}\gets (z_i^{0})^{i=1,\ldots,N_1}$\;
	$(\lambda_i^{0,\mathrm{corr}}
	)_{i=1,\ldots,N_1}\gets 
	(\lambda^0_i)_{i=1,\ldots,N_1}$\;
	\For(\tcp*[f]{Update of controls and trajectories}){$l=1,\ldots,M$ 
	}{

		$u_l^{\mathrm{new}} \gets
\arg \max_{ v\in \R^k}
\left\{
\sum_{i=1}^{N_1}  \left(
\lambda^{l-1,\mathrm{corr}}_{i}\cdot 
F(z_{i}^{l-1,\mathrm{new}})\cdot v \right) 
- \frac{\beta}{2}|v|^2_2
-\frac{1}{2\rho}|v-u_l|^2_2
\right\}$\;
	\For{$i=1,\ldots,N_1$}{
	Compute
	$z_i^{l,\mathrm{new}}$ using 
	$z_{i}^{l-1,\mathrm{new}}$ and 
	$u_l^{\mathrm{new}}$\;
	$\lambda_i^{l, \mathrm{corr}}
	\gets \textcolor{red}{\lambda^l_i } 
	+ \sum_{ j=1}^{N_2} \gamma_N^{i,j}(z_i^l-y_j) 
	- \sum_{ j=1}^{N_2} \gamma_N^{i,j}(z_i^{l, \mathrm{new}}-y_j)$\;}
		}

	$\mathrm{Cost^{new}}\textcolor{red}{\gets}
\sum_{i=1,\ldots,N_1}^{j=1,\ldots,N_2} \gamma_N^{i,j}|z^{M, \mathrm{new}}_i-y_j|_2^2
    + \frac\beta2 \| u \|_{L^2}^2$\;
	\eIf(\tcp*[f]{Backtracking for $\rho$}){$\mathrm{Cost}> \mathrm{Cost^{new}}$ }{
		$u\gets u^{\mathrm{new}}$,
		$z\gets z^{\mathrm{new}}$\;
		$\mathrm{Cost}\gets \mathrm{Cost^{new}}$\;
		$\mathrm{flag} \gets 1$\;
		}
		{
		$\gamma \gets \tau \gamma$\;
		$\mathrm{flag} \gets 0$\;
		}
	}
\caption{Iterative Maximum Principle}
\label{alg:iter_PMP}
\end{algorithm}

\begin{remark}\label{rmk:corr_PMP}
    The correction for the value of the covector at the line 20 of Algorithm~\ref{alg:iter_PMP} is not present in the original scheme proposed in \cite{SS80}, where the authors considered optimal control problems without end-point cost.
\end{remark}

\begin{remark} \label{rmk:max_hamilt}
    The maximization of the augmented Hamiltonian in line~17 of Algorithm~\ref{alg:iter_PMP} is a rather inexpensive step, since we have to deal with a quadratic function whose Hessian is diagonal. This is a beneficial consequence of the linear-control dynamics, resulting in the fact that the first term of the augmented Hamiltonian is linear in $v$ (see again line~17).
    In the case of a standard neural ODE, we would have  $\arg \max_{ v\in \R^k}
\left\{
\sum_{i=1}^{N_1}  \left(
\lambda^{l-1,\mathrm{corr}}_{i}\cdot 
G(z_{i}^{l-1,\mathrm{new}}, v) \right) 
- \frac{\beta}{2}|v|^2_2
-\frac{1}{2\rho}|v-u_l|^2_2
\right\}$, resulting in a non-quadratic (and potentially non-concave) maximization problem, whose resolution may be expensive.
\end{remark}

\begin{remark} \label{rmk:grad_method}
    As an alternative, it is possible to address the minimization of the cost functional $\F^{N,\beta}:\U\to \R$ using a gradient flow approach. Namely, it is possible to project the gradient field induced by $\F^{N,\beta}$ onto the finite dimensional subspace $\U_M$.
    We recall that in \cite{S_grad} the gradient flows related to linear-control problems have been studied theoretically, while in \cite{S_deep, S_ens} the gradient-based algorithm outlined above has been implemented and tested.
    In general, it has slightly worse per-iteration performances than the PMP-based algorithm, but it is more suitable for parallel computations.
\end{remark}


\subsection{A numerical experiment}

We present here a numerical experiment in $\R^2$ that we used to validate our approach.
In this case, we considered as reference measure $\mu$ the uniform probability measure supported in the disc centered at the origin and with radius $R=0.5$, and we constructed $\mu_N$ with a uniform triangulation of $\supp(\mu)$ with size $0.04$, resulting in $571$ equally-weighted atoms (see Figure~\ref{fig:Experiments}).
Then, we took the convex function $f:\R^2 \to\R$ defined as 
\begin{equation*}
    f(x) = \sqrt{(x-v)^\top Q (x-v) + 2}, 
    \quad v= \left( \begin{matrix}
        0.5\\
        0.5
    \end{matrix} \right)
    \quad Q= \left( \begin{matrix}
        3 & 1\\
        1 & 2
    \end{matrix} \right),
\end{equation*}
and we set $T:=\nabla_x f$. 
Then, we defined $\nu:= T_\# \mu$, and we obtained the empirical measure $\nu_N$ by sampling $1500$ i.i.d. data-points from $\mu$, and by transforming them using $T$. In this way, we got $1500$ independent samplings from $\nu$.  
At this point, we used the Python package \cite{POT_library} to compute the optimal transport plan $\gamma_N = (\gamma_N^{i,j})^{j=1,\ldots,N_2}_{i=1,\ldots,N_1}$. Since the problem has modest dimensions, we used the non-regularized solver, and we observed that every optimal transport plan computed satisfied the sparsity bound investigated in \cite{AV22}.
Using the vector fields that had been reported to be the best-performing in \cite{S_deep}, we dealt with the following linear-control system
on the time interval $[0,1]$:
\begin{equation} \label{eq:ctrl_sys_exper}
\begin{split}
\dot x = 
\left(
\begin{matrix}
u_1\\
u_2
\end{matrix}
\right)&
+ e^{-\frac{1}{2\zeta}|x|^2}
\left(
\begin{matrix}
u_1'\\
u_2'
\end{matrix}
\right)
+ \left(
\begin{matrix}
u_1^1 & u^2_1\\
u_2^1 & u_2^2
\end{matrix}
\right)
\left(
\begin{matrix}
x_1\\
x_2
\end{matrix}
\right)
\\ & \qquad+
e^{-\frac{1}{2\zeta}|x|^2}
\left(
\begin{matrix}
u_1^{1,1}x_1^2 + u_1^{1,2}x_1x_2 + u_1^{2,2}x_2^2\\
u_2^{1,1}x_1^2 + u_2^{1,2}x_1x_2 + u_2^{2,2}x_2^2
\end{matrix}
\right),
\end{split}
\end{equation}
where we set $\zeta =10$. We divided the time horizon $[0, 1]$ into $32$ equally-spaced subintervals, corresponding to the discretization step-size $h=2^{-5}$ for \eqref{eq:ctrl_sys_exper}. Finally, we set $\beta = 5\cdot 10^{-4}$ in \eqref{eq:funct_real_problems}, and we minimized $\F^{N,\beta}$ using Algorithm~\ref{alg:iter_PMP}, in order to construct a flow $\Phi_{u}$ of \eqref{eq:ctrl_sys_exper} that could serve as an approximation of $T$.
The results are reported in Figure~\ref{fig:Experiments}.\\
As we can see, the transformed measure $\Phi_{u\, \#}\mu_N$ managed to find correctly the boundary and the shape of the target empirical measure $\nu_N$, as well as the fact that the mass is not uniformly spread over the support of the target measure. 
Finally, in the last picture, we compared $T_\#\mu_N$ and $\Phi_{u \,\#}\mu_N$, i.e., the transformation of the uniform grid over the reference disc through the correct optimal transport map and the computed approximation, respectively, resulting in an accurate reconstruction.

\begin{figure}
    \centering
    \includegraphics[scale=0.44]{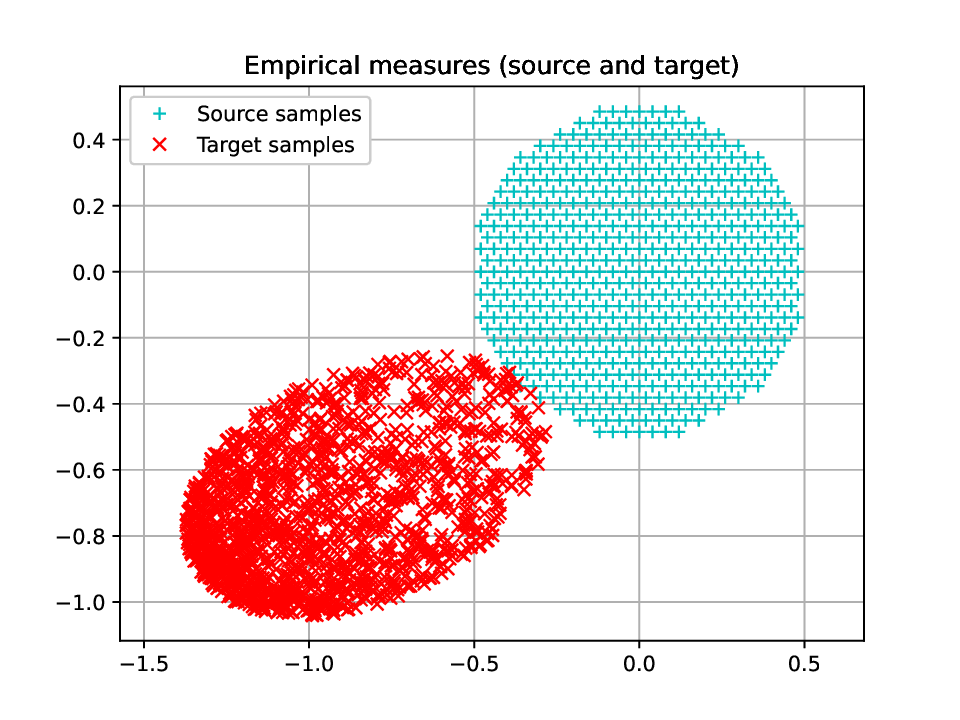}
    \includegraphics[scale=0.44]{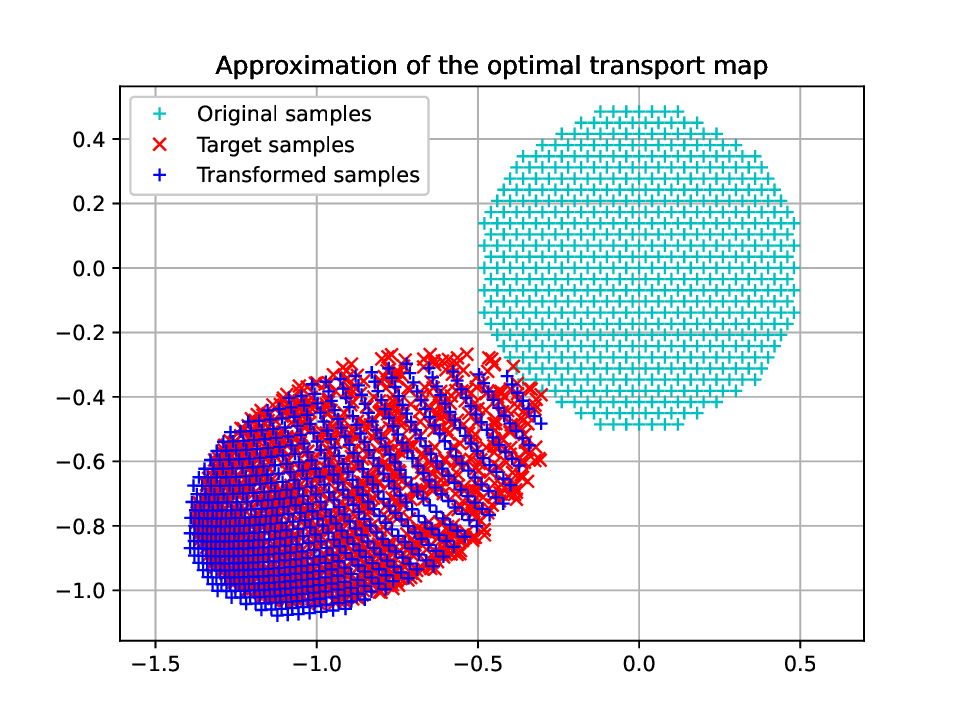}\\
    \includegraphics[scale=0.38]{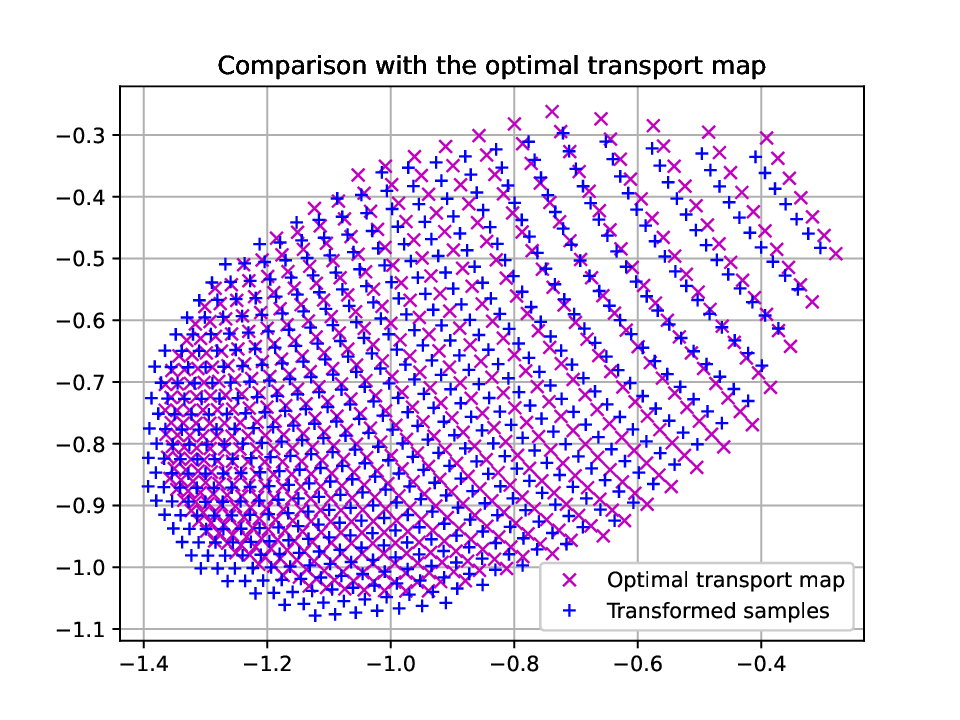}
    \caption{Approximation of the optimal transport map using samplings of the transported measure.}
    \label{fig:Experiments}
\end{figure}


\section*{Conclusions}
In this paper, we investigated the possibility of recovering the $W_2$-optimal transport map between $\mu,\nu$ as flows of linear-control neural ODEs.
We first showed that, under appropriate hypotheses on the measures $\mu,\nu$, the optimal transport map $T$ is a diffeomorphism isotopic to the identity (see Proposition~\ref{prop:isotopic}).
Hence, leveraging on the expressivity results for linear-control systems established in \cite{AS1,AS2}, in Corollary~\ref{cor:ot_map_approx} we proved that it is possible to approximate $T$ in the $C^0$-norm by means of flows of linear-control systems.
Then, we consider the case where only discrete approximations $\mu_N,\nu_N$ of $\mu,\nu$ are available, and we used a discrete $W_2$-optimal coupling $\gamma_N$ between $\mu_N,\nu_N$ to define the functional $\F^{N,\beta}$.
Then, in Theorem~\ref{thm:OT_map_appr} we proved that, if $\mu_N\weak^* \mu$ and $\nu_N\weak^* \nu$ as $N\to\infty$, then the optimal control problems involving $\F^{N,\beta}$ are $\Gamma$-convergent to a limiting functional, that concerns the approximation of $T$ in the $L^2_\mu$-norm.
Finally, we proposed an iterative algorithm based on the Pontryagin Maximum Principle for minimizing $\F^{N,\beta}$, resulting in a scheme for producing a normalizing flow.
Finally, we tested the method on an example in $\R^2$.


\subsection*{Acknowledgments} 
A.S. acknowledges partial support from INdAM--GNAMPA.
A.S. wants to thank Ismael Medina for the helpful discussions.
{The majority of this work was developed while S.F. was supported by the Lagrange Mathematics and Computation Research Center which S.F. aknowledges.}
S.F. benefits from the support of MUR (PRIN project 202244A7YL) and thanks DipE awarded to
the DIMA-Unige (CUP D33C23001110001).
S.F. wishes to thank Luca Nenna and Simone Di Marino for related discussions and suggestions.

\appendix
\section{Proofs of Section~\ref{subsec:lin-ctrl}} \label{app:lin-ctrl}
Here we prove the intermediate results needed to establish Proposition~\ref{prop:conv_flows}.
We first recall a version of the version of the Gr\"onwall-Bellman inequality.

\begin{lemma}[Gr\"onwall-Bellman Inequality] 
\label{lem:Gron}
Let $f:[a,b]\to\R_+$ be a
non-negative continuous function
and let us assume that there exists
a constant $\alpha>0$ and a non-negative 
function $\beta\in L^1([a,b],\R_+)$
such that
\[
f(s) \leq \alpha + \int_a^s\beta(\tau)f(\tau) \,d\tau
\]
for every $s\in[a,b]$. Then, for every 
$s\in[a,b]$ the following inequality holds:
\begin{equation} \label{eq:Gron_ineq}
f(s) \leq \alpha e^{\|\beta\|_{L^1}}.
\end{equation}
\end{lemma}
\begin{proof}
This statement follows as a particular case
of \cite[Theorem~5.1]{EK86}. 
\end{proof}

We remind that from the Jensen inequality it follows that
\begin{equation} \label{eq:L1_L2}
    \|u\|_{L^1}:=
    \int_0^1 \sum_{i=1}^k|u_i(t)|\,dt
    \leq \sqrt{k} \|u\|_{L^2}
\end{equation}
for every $u\in\U=L^2([0,1],\R^k)$.
In the next result we show that the flows generated by controls that are equi-bounded in $L^2$ are in turn equi-bounded on compact subsets of $\R^n$.

\begin{lemma} \label{lem:bound_Phi}
For every $u\in\U$, let $\Phi_u:\R^n\to\R^n$ be the 
flow defined as in \eqref{eq:def_Phi_u},
associated to the linear-control system 
\eqref{eq:lin_ctrl} and corresponding to
the admissible control $u$.
Then, for every $r>0$ and for every $\rho>0$ there exists
$R>0$ such that
\begin{equation} \label{eq:bound_Phi}
    |\Phi_u(x)|_2 \leq R
\end{equation}
for every $x\in\R^n$ satisfying $|x|_2\leq r$ 
and for every $u\in\U$ with $\|u\|_{L^2}\leq \rho$.
\end{lemma}
\begin{proof}
Let $u\in\U$ be an admissible control and 
let $x\in \R^n$ be the Cauchy datum for 
the initial-value problem \eqref{eq:Cau_ctrl}.
If we consider the curve $x_u:[0,1]\to\R^n$
that solves the Cauchy problem \eqref{eq:Cau_ctrl},
then from the sub-linear growth inequality
\eqref{eq:sub_growth} it descends that
\begin{align*}
    |x_u(t)|_2 &\leq |x|_2 +\int_0^t 
    \sum_{i=1}^k|F_i(x_u(s))|_2|u_i(s)|\, ds\\
    & \leq |x|_2 +\int_0^t 
    C(|x_u(s)|_2 +1 )
    \sum_{i=1}^k|u_i(s)|\, ds \\
    &\leq |x|_2 + \sqrt{k}C \|u\|_{L^2}
    + C\int_0^1 |x_u(s)|_2 \sum_{i=1}^k|u_i(s)|\, ds
\end{align*}
for every $t\in[0,1]$,
where we used \eqref{eq:L1_L2} in the last passage.
In virtue of Lemma~\ref{lem:Gron}, the previous 
inequality yields
\begin{equation*}
    |x_u(t)|_2 \leq \left( |x|_2 + C\sqrt{k} \|u\|_{L^2}
    \right)
    e^{\sqrt k \|u\|_{L^2}}
\end{equation*}
for every $t\in[0,1]$. In particular, using $t=1$ 
in the last inequality and setting
$R:=(r+C \sqrt k \rho)e^{\sqrt k \rho}$, we deduce
\eqref{eq:bound_Phi}.
\end{proof}

We report below the proof of Lemma~\ref{lem:Lip_flow}.

\begin{proof}[Proof of Lemma~\ref{lem:Lip_flow}]
Let $u\in\U$ be an admissible control, and let
us consider $x^1,x^2 \in \R^n$.
Let $x_u^1,x_u^2:[0,1]\to\R^n$ be the solutions of
the Cauchy problem
\eqref{eq:Cau_ctrl} corresponding to the control
$u$ and to the initial data $x^1,x^2$, respectively.
Then, using the Lipschitz-continuity condition
\eqref{eq:lipsch_fields}, we compute
\begin{align*}
    |x_u^1(t)-x_u^2(t)|_2  &\leq
    |x^1-x^2|_2 + \int_0^t 
    \sum_{i=1}^k|F_i(x_u^1(s))-F_i(x_u^2(s))|_2 |u_i(s)|
    \,ds \\
    &\leq |x^1-x^2|_2 + L\int_0^t 
    |x_u^1(s)-x_u^2(s)|_2\sum_{i=1}^k |u_i(s)|
    \,ds
\end{align*}
for every $t\in[0,1]$. Owing to Lemma~\ref{lem:Gron}
and \eqref{eq:L1_L2},
we deduce that
\begin{equation*}
    |x^1_u(t)-x_u^2(t)|_2 \leq e^{L\sqrt k \|u\|_{L^2}}|x^1 - x^2|_2
\end{equation*}
for every $t\in[0,1]$. In particular, setting $t=1$
in the last inequality, we obtain that
\begin{equation}
    |\Phi_u(x^1)-\Phi_u(x^2)|_2 \leq 
    e^{L\sqrt k \rho}|x^1-x^2|_2
\end{equation}
for every $x^1,x^2\in \R^n$ and for every $u\in \U$
such that $\|u\|_{L^2}\leq \rho$.
This proves \eqref{eq:Lip_flow}.
\end{proof}

\begin{proof}[Proof of Proposition~\ref{prop:conv_flows}]
Let $K\subset \R^n$ be a compact set.
For every $x\in K$ and for every 
$m\in \NN\cup\{\infty \}$, let $x_{u_m}:[0,1]\to\R^n$
be the solution of the Cauchy problem \eqref{eq:Cau_ctrl}
corresponding to the admissible control 
$u_m$ and with initial datum $x_{u_m}(0) =x$.
In virtue of \cite[Lemma~7.1]{S_grad}, we have that
\begin{equation*}
    \lim_{m\to\infty} \sup_{t\in[0,1]}
    |x_{u_m}(t)-x_{u_\infty}(t)|_2 =0,
\end{equation*}
which in particular implies the point-wise convergence
\begin{equation} \label{eq:point_conv_phi}
    \lim_{m\to\infty}|\Phi_{u_m}(x)-\Phi_{u_\infty}(x)|_2
    =0
\end{equation}
for every $x\in K$.
From the weak convergence
$u_m\weak_{L^2}u_\infty$ as $m\to\infty$, we deduce that
there exists $\rho>0$ such that
\begin{equation} \label{eq:bound_L2_conv_seq}
    \sup_{m\in \NN\cup\{\infty \}}
    \|u_m\|_{L^2} \leq \rho.
\end{equation}
Combining \eqref{eq:bound_L2_conv_seq} with 
Lemma~\ref{lem:bound_Phi}, we obtain that there
exists $R>0$ such that
\begin{equation} \label{eq:equi_bound}
    \sup_{x\in K}|\Phi_{u_m}(x)|_2 \leq R
\end{equation}
for every $m\in \NN\cup\{ \infty\}$. Moreover, from \eqref{eq:bound_L2_conv_seq} and Lemma~\ref{lem:Lip_flow} it follows that there exists $L'>0$ such that
\begin{equation} \label{eq:equi_Lip}
    |\Phi_{u_m}(x^1)-\Phi_{u_m}(x^2)|_2
    \leq L'|x^1-x^2|_2
\end{equation}
for every $x^1,x^2\in K$ and for every $m\in \NN \cup\{ \infty \}$. Therefore, if we consider the restrictions $\Phi_{u_m}|_K:K\to\R^n$ for every $m\in \NN\cup\{\infty\}$, from \eqref{eq:equi_bound}-\eqref{eq:equi_Lip} we deduce that the sequence of the restricted flows $(\Phi_{u_m}|_K)_{m\in\NN}$ is equi-bounded and equi-Lipschitz. 
Then, applying
Arzel\`a-Ascoli Theorem (see, e.g., \cite[Theorem~4.25]{B11}), we deduce that $(\Phi_{u_m}|_K)_{m\in\NN}$ is pre-compact with respect  to the uniform convergence. On the other hand, the point-wise convergence \eqref{eq:point_conv_phi} guarantees that the set of cluster elements of the sequence $(\Phi_{u_m}|_K)_{m\in\NN}$ is reduced to $\{ \Phi_{u_\infty}|_K \}$. This proves \eqref{eq:conv_flows} and concludes the proof.
\end{proof}


\bigskip
\bigskip
\bigskip


\begin{thebibliography}{99}

\bibitem{AgCapo}
A. Agrachev, M. Caponigro.
Controllability on the group of diffeomorphisms.
\textit{Ann. Inst. Henri Poincare (C) Anal. Non Lin\'eaire},
26(6): 2503-2509 (2009).
doi:~10.1016/j.anihpc.2009.07.003

\bibitem{AS04} A. Agrachev, Y. Sachkov.
\textit{Control Theory from the Geometric Viewpoint.}
Springer-Verlag Berlin Heidelberg (2004).


\bibitem{AS1}
A. Agrachev, A. Sarychev.
Control in the spaces of ensembles of points.
\textit{SIAM J. Control Optim.}, 58: 1579-1596 (2020)  
doi: 10.1137/19M1273049

\bibitem{AS2}
A. Agrachev, A. Sarychev.
Control on the manifolds of mappings with a view to the deep learning.
\textit{J. Dyn. Control Syst.}, 28: 989–1008 (2022).
doi: 0.1007/s10883-021-09561-2



\bibitem{AG}
L. Ambrosio, N. Gigli.
\textit{A users's guide to optimal transport.} 
In: Modelling and Optimisation of Flows on Networks Lecture Notes in Mathematics, Springer Berlin Heidelberg,
1–155 (2013).

\bibitem{ATTY}
S. Arguill\`ere, E. Tr\'elat, A. Trouv\'e, L. Younes.
Shape deformation analysis from the optimal control viewpoint.
\textit{J. Math. Pures Appl.},
104(1): 139-178 (2015).
doi:~10.1016/j.matpur.2015.02.004

\bibitem{ACB}
M. Arjovsky, S. Chintala, L. Bottou. 
Wasserstein generative adversarial network. 
\textit{International Conference on Learning Representations (ICML)} pp. 214–223, (2017).

\bibitem{ABGV_18}
G. Auricchio, F. Bassetti, S. Gualandi, M. Veneroni.  
Computing Kantorovich-Wasserstein Distances on $d$-dimensional histograms using $(d+1)$-partite graphs. 
\textit{Adv. Neural Inf. Process. Syst.}, 31 (2018).


\bibitem{AV22}
G. Auricchio, M. Veneroni. 
On the Structure of Optimal Transportation Plans between Discrete Measures. 
\textit{Appl. Math. Optim.}, 85(42) (2022). 
doi: 10.1007/s00245-022-09861-4




\bibitem{BSR18}
M. Blondel, V. Seguy, A. Rolet. 
Smooth and sparse optimal transport.
\textit{Proceedings of the 21st AISTATS, PMLR}, 84: 880–889 (2018).


\bibitem{BCFH_23}
B. Bonnet, C. Cipriani, M. Fornasier, H. Huang.
A measure theoretical approach to the mean-field maximum principle for training NeurODEs. 
\textit{Nonlinear Analysis}, 227:113-161 (2023).
doi: 10.1016/j.na.2022.113161



\bibitem{B11}
H. Brezis.
\textit{Functional Analysis, Sobolev Spaces and Partial Differential Equations.}
Universitext, 
Springer New York NY (2011).
doi: 10.1007/978-0-387-70914-7

\bibitem{Brualdi}
R.A. Brualdi.
\textit{Combinatorial Matrix Classes.}
Cambridge University Press (2006).
doi: 10.1017/CBO9780511721182

\bibitem{Ca1}
L.A. Caffarelli,
Some regularity properties of solutions of {M}onge {A}mp\`ere equation.
\textit{Comm. Pure Appl. Math.}, 44 (1991).
doi: 10.1002/cpa.3160440809


\bibitem{Ca2}
L.A. Caffarelli,
Boundary regularity of maps with convex potentials--{II}.
\textit{Ann. Math.} 144(3): 453-496 (1996).
doi: 10.2307/2118564

\bibitem{CR12}
G. Canas, L. Rosasco. 
Learning probability measures with respect to optimal transport metrics. 
\textit{Adv. Neur. Inf. Process Syst.},
25 (2012).


\bibitem{Capo}
M. Caponigro. 
Families of vector fields which generate the group of diffeomorphisms. 
\textit{Proc. Steklov Inst. Math.}, 
270: 141–155 (2010). 
doi: 10.1134/S0081543810030107

\bibitem{CRBD18}
R.T.Q. Chen, Y. Rubanova, J. Bettencourt, D.K. Duvenaud.
{Neural Ordinary Differential Equations}. 
\textit{Adv. Neur. Inf. Process Syst.},
31 (2018).


\bibitem{CFS_23}
C. Cipriani, M. Fornasier, A. Scagliotti.
From NeurODEs to AutoencODEs: a mean-field control framework for width-varying neural networks. 
\textit{Eur. J. Appl. Math.}, published online (2024).
doi: 10.1017/S0956792524000032


\bibitem{Cut13}
M. Cuturi.
Sinkhorn Distances: Lightspeed Computation of Optimal Transport.
\textit{Adv. Neur. Inf. Process Syst.},
26 (2013).

\bibitem{D93}
G. Dal Maso. 
{\it An Introduction to $\Gamma$-convergence.}
Progress in nonlinear differential equations and their applications, Birkh\"auser
Boston MA (1993).

\bibitem{LGL_21}
L. De Lara, A. González-Sanz, J.M. Loubes. 
A Consistent Extension of Discrete Optimal Transport Maps for Machine Learning Applications. 
\textit{arXiv preprint}, arXiv:2102.08644 (2021).


\bibitem{DPHF}
G. De Philippis, A. Figalli.
{The Monge--Amp{\`e}re equation and its link to optimal transportation.} 
\textit{Bull. Am. Math. Soc.},
51(4): 527-580 (2014). 


\bibitem{EGPZ}
C. Esteve, B. Geshkovski, D. Pighin, E. Zuazua.
{Large-time asymptotics in Deep Learning}. 
\textit{arXiv preprint}, arXiv:2008.02491 (2020).

\bibitem{EK86}
S. Ethier, T. Kurtz.
{\it Markov Processes: Characterization and Convergence.} Wiley series in probability and 
statistics,
John Wiley \& Sons New York (1986). 

\bibitem{POT_library}
R. Flamary et al.
POT Python Optimal Transport library,
\textit{J. Mach. Learn. Res.}, 22(78): 1-8 (2021).

\bibitem{FHS23}
M. Fornasier, P. Heid, G.E. Sodini.
Approximation Theory, Computing, and Deep Learning on the Wasserstein Space
\textit{arXiv preprint}, arXiv:2310.19548
(2023).


\bibitem{WE17}
W. E.
A proposal on machine learning via dynamical systems.
\textit{Comm. Math. Stat.}, 1(5):1-11  (2017).
doi: 10.1007/s40304-017-0103-z



\bibitem{FGOP_20}
C. Finlay, A. Gerolin, A.M. Oberman, A.A. Pooladian. 
Learning normalizing flows from Entropy-Kantorovich potentials.
\textit{arXiv preprint}, arXiv:2006.06033 (2020).

\bibitem{haber17}
E.~Haber, L.~Ruthotto.
Stable architectures for deep neural networks.
\textit{Inverse problems}, 34(1) (2018).
doi: 10.1088/1361-6420/aa9a90





\bibitem{H80} 
J. Hale.
{\it Ordinary Differential Equations.}
Krieger Publishing Company (1980).


\bibitem{Rigollet_21}
J.-C. Hütter, P. Rigollet.
Minimax estimation of smooth optimal transport maps.
\textit{Ann. Stat.}, 
49(2):1166-1194 (2021).
doi: 10.1214/20-AOS1997


\bibitem{Koby_nf_20}
I. Kobyzev, S.J. Prince, M.A. Brubaker.
Normalizing flows: An Introduction and Review of Current Methods.
\textit{IEEE Trans. Pattern Anal. Mach. Intell.},
43(11):3964-3979 (2020).
doi: 10.1109/TPAMI.2020.2992934


\bibitem{Thorpe_17}
S. Kolouri, S.R. Park, M. Thorpe, D. Slepcev, G.K. Rohde.
Optimal mass transport: Signal processing and machine-learning applications. 
\textit{IEEE signal processing magazine},
34(4):43–59, (2017).
doi: 10.1109/MSP.2017.2695801

\bibitem{kuehn_23}
C. Kuehn, S.-V. Kuntz.  
Embedding Capabilities of Neural ODEs. 
\textit{arXiv preprint}, arXiv:2308.01213 (2023).

\bibitem{LMM21}
D.A. Lorenz, P. Manns, C. Meyer. 
Quadratically Regularized Optimal Transport. 
\textit{Appl. Math. Optim.}, 83: 1919–1949 (2021). 
doi: 10.1007/s00245-019-09614-w


\bibitem{MBNW_21}
T. Manole, S. Balakrishnan, J. Niles-Weed, L. Wasserman. 
Plugin estimation of smooth optimal transport maps. 
\textit{Ann. Stat.}, 52(3): 966-998 (2024).
doi: 10.1214/24-AOS2379


\bibitem{MSS21} 
Q. M\'{e}rigot, F. Santambrogio, C. Sarrazin.
Non-asymptotic convergence bounds for Wasserstein approximation using point clouds.
\textit{Adv. Neur. Inf. Process Syst.},
34:12810-12821 (2021).


\bibitem{MDBCR_23}
G. Morel, L. Drumetz, S. Bena\"ichouche, N. Courty, F. Rousseau.
Turning Normalizing Flows into Monge Maps with Geodesic Gaussian Preserving Flows. 
\textit{Transactions on Machine Learning Research}, (2023).

\bibitem{Papa_nf_21}
G. Papamakarios, E. Nalisnick, D.J. Rezende, S. Mohamed, B. Lakshminarayanan. 
Normalizing flows for probabilistic modeling and inference. 
\textit{J. Mach. Learn. Res.}, 22(1):2617-2680 (2021).


\bibitem{PCFH_16}
M. Perrot, N. Courty, R. Flamary, A. Habrard.
Mapping estimation for discrete optimal transport. 
\textit{Adv. Neur. Inf. Process Syst.}, 29
(2016).

\bibitem{PeyCut}
G. Peyr\'e, M. Cuturi.
\textit{Computational Optimal Transport: With Applications to Data Science.}
Foundations and Trends in Machine Learning\textregistered, 11(5-6):355-607 (2019). 
doi: 10.1561/2200000073


\bibitem{RKB}
L. Rout, A. Korotin, E. Burnaev. 
Generative modeling with optimal transport maps.
\textit{International Conference on Learning Representations (ICLR)}, (2022).

\bibitem{RBZ_nODE}
D. Ruiz-Balet, E. Zuazua. 
Neural ODE control for classification, approximation, and transport.
\textit{SIAM Review}, 65(3):735-773 (2023).
doi: 10.1137/21M141143

\bibitem{RBZ_NF}
D. Ruiz-Balet, E. Zuazua. 
Control of neural transport for normalizing flows. 
\textit{J. Math. Pures Appl.}, (2023).
doi: 10.1016/j.matpur.2023.10.005

\bibitem{SS80}
Y. Sakawa, Y. Shindo.
{On global convergence of an algorithm for optimal 
control.}
{\it IEEE Trans. Automat. Contr.},
25(6):1149-1153 (1980).

\bibitem{Sant} F. Santambrogio.
\textit{Optimal transport for applied mathematicians.}
{Birk\"auser, NY} (2015).

\bibitem{S_grad} A. Scagliotti.
{A gradient flow equation for optimal control problems
with end-point cost.}
\textit{J. Dyn. Control Syst.}, 29: 521–568 (2023). 
doi: 10.1007/s10883-022-09604-2

\bibitem{S_deep} 
A. Scagliotti.
{Deep Learning approximation of diffeomorphisms via
linear-control systems.}
{\it Math. Control Rel. Fields}, 13(3): 1226-1257 (2023). 
doi: 10.3934/mcrf.2022036.

\bibitem{S_ens} A. Scagliotti.
{Optimal control of ensembles of dynamical systems.}
{\it ESAIM: Control, Optim., Calc. Var.}, (2023).
doi: 10.1051/cocv/2023011

\bibitem{Via12}
FX. Vialard, L. Risser, D. Rueckert, C.J. Cotter.
Diffeomorphic 3D Image Registration via Geodesic Shooting Using an Efficient Adjoint Calculation. 
\textit{Int J Comput Vis} 
97: 229–241 (2012). 
doi:~10.1007/s11263-011-0481-8


\bibitem{Vi}
C.Villani.
{\em Optimal transport: old and new},
Vol 338, Springer-Verlag, Berlin (2008).


\bibitem{W34}H. Whitney.
{Analytic extensions of differentiable functions defined in closed sets.}
\textit{Trans. Amer. Math. Soc.}, 
36: 63-89 (1934).

\bibitem{Y19}
L. Younes.
\textit{Shapes and Diffeomorphisms (II edition).}
Springer Berlin, Heidelberg (2019).
doi:~10.1007/978-3-662-58496-5



\end{thebibliography}
\end{document}